\documentclass[a4paper,11pt]{amsart}
\usepackage{amssymb}
\usepackage{latexsym}
\usepackage{amsmath}
\usepackage{enumerate}
\usepackage{amsmath, hyperref}
\usepackage{geometry}
\usepackage{tikz}
\usepackage[all,pdf]{xy}
\newtheorem{theorem}{Theorem}[section]
\newtheorem{lemma}[theorem]{Lemma}

\newtheorem{proposition}[theorem]{Proposition}

\theoremstyle{definition}
\newtheorem{definition}[theorem]{Definition}
\newtheorem{question}[theorem]{Question}
\newtheorem{fact}[theorem]{Fact}
\newtheorem{example}[theorem]{Example}

\newtheorem{remark}[theorem]{Remark}

\numberwithin{equation}{section}

\title{On the relationships between some meta-mathematical properties of arithmetical theories}

\author{Yong Cheng}

\thanks{I thank Albert Visser for enlightening discussions. The proof ideas of Theorem \ref{thm on Succ} and Theorem \ref{zero sharp does not imply creative} are from him and the proof details are due to the author. I am grateful to Albert Visser for his kind permission to use these ideas. I also thank  helpful comments  for improvements from the referees.}



\begin{document}

\begin{abstract}
In this work, we aim at understanding incompleteness in an abstract way via metamathematical
properties of formal theories. We systematically examine the relationships between the following twelve important metamathematical properties of arithmetical theories: {\sf Rosser}, {\sf EI} (Effectively inseparable), {\sf RI} (Recursively inseparable), {\sf TP} (Turing persistent), {\sf EHU} (essentially hereditarily undecidable), {\sf EU} (essentially undecidable), {\sf Creative}, $\mathbf{0}^{\prime}$ (theories with Turing degree $\mathbf{0}^{\prime}$), {\sf REW} (all RE sets are weakly representable), {\sf RFD} (all recursive functions are definable), {\sf RSS} (all recursive sets are strongly representable), {\sf RSW} (all recursive sets are weakly representable). Given any two properties $P$ and $Q$ in the above list, we examine whether $P$ implies $Q$.
\end{abstract}

\maketitle

\section{Introduction}

Since G\"{o}del, research on incompleteness has greatly deepened our understanding of the incompleteness phenomenon. In this work, we aim at understanding incompleteness in an abstract way via metamathematical
properties of formal theories. We study metamathematical properties of RE theories which exhibit behaviors that can be related to the incompleteness/undecidability. In particular, we discuss  the relationships between  twelve important metamathematical properties of arithmetical theories in the literature as defined below.

All theories are supposed to be first-order  RE (recursively enumerable) theories. We equate a theory with the set of theorems provable in it. We always assume the arithmetization of the base theory. Given a sentence $\phi$, let $\ulcorner\phi\urcorner$ denote the G\"{o}del number of $\phi$.
Under arithmetization, we equate a set  of sentences  with the set of G\"{o}del numbers of those sentences.
In the following, we introduce the twelve meta-mathematical properties of arithmetical theories as follows: {\sf Rosser}, {\sf EI}, {\sf RI}, {\sf TP}, {\sf EHU}, {\sf EU}, {\sf Creative}, $\mathbf{0}^{\prime}$, {\sf REW}, {\sf RFD}, {\sf RSS}, {\sf RSW}.

We first introduce the notions of {\sf Rosser} and {\sf EI} theories.
In the
following definition, we denote the RE set with index $i$ by $W_i$ where $W_i=\{x: \exists y \, T_1(i,x, y)\}$ and $T_1(z,x, y)$ is the Kleene predicate (see \cite{Kleene}). We  recall that a pair $(A,B)$ of disjoint RE sets is  effectively inseparable $(\sf EI)$ if there is a recursive function $f(x,y)$ such that for any $i$ and $j$, if $A\subseteq W_i$ and $B\subseteq W_j$ with $W_i\cap W_j=\emptyset$, then $f(i,j)\notin W_i\cup W_j$.

\begin{definition}[Rosser theories, the nuclei of a theory, {\sf EI} theories]\label{def of 1.1}
Let $T$ be a consistent RE theory, and $(A,B)$ be a disjoint pair of RE sets.
\begin{enumerate}[(1)]
\item We say $(A,B)$ is \emph{separable} in $T$ if there is a formula $\phi(x)$ with only one free variable such that if $n\in A$, then $T\vdash \phi(\overline{n})$, and if $n\in B$, then $T\vdash \neg\phi(\overline{n})$.
\item We say $T$ is  \emph{Rosser}  if any disjoint  pair of RE sets is separable in $T$.
    \item The pair $(T_P, T_R)$ are called the \emph{nuclei} of a theory $T$, where  $T_P$ is the set of G\"{o}del numbers of sentences provable in $T$, and $T_R$ is the set of G\"{o}del numbers of sentences refutable in $T$ (i.e., $T_P=\{\ulcorner\phi\urcorner: T\vdash\phi\}$ and $T_R=\{\ulcorner\phi\urcorner: T\vdash\neg\phi\}$).

  \item We say $T$ is \emph{effectively inseparable} $(\sf EI)$ if $(T_P, T_R)$ is $\sf EI$.
\end{enumerate}
\end{definition}

Now we introduce the notions of $\sf RI$, {\sf TP} and {\sf EHU} theories.
Given a disjoint pair  $(A,B)$ of RE sets, we say $(A, B)$ is \emph{recursively inseparable} ($\sf RI$)   if  there is no recursive set $X\subseteq \mathbb{N}$ such that $A\subseteq X$ and $X\cap B=\emptyset$. The notion of Turing reducibility $({\sf \leq_T})$ is standard, see \cite[p.137]{Rogers87}.

\begin{definition}[$\sf RI$, {\sf TP} and {\sf EHU} theories]\label{def of 1.2}
Let $T$ be a consistent RE theory, and $(A,B)$ be a disjoint pair of RE sets.
\begin{enumerate}[(1)]
      \item We say $T$ is \emph{recursively inseparable} $(\sf RI)$ if $(T_P, T_R)$ is $\sf RI$.
\item We say $T$ is \emph{Turing persistent} ({\sf TP}) if $T$ is undecidable and for any consistent RE extension $S$ of $T$, $T {\sf \leq_T} S$.
 \item We say $T$ is \emph{hereditarily undecidable} ({\sf HU}) if every sub-theory $S$ of $T$ over the same language  is undecidable.
       \item We say $T$ is \emph{essentially hereditarily undecidable} ({\sf EHU}) if any consistent RE extension of $T$ is {\sf HU}.
\end{enumerate}
\end{definition}

Now we introduce the notions of {\sf EU}, {\sf Creative}, and $\mathbf{0}^{\prime}$ theories.
We recall that $A\subseteq\mathbb{N}$ is \emph{productive} if there exists a recursive function $f(x)$ (called a productive function for $A$) such that for every number $i$, if $W_i\subseteq A$, then $f(i)\in A- W_i$; and $A\subseteq\mathbb{N}$ is \emph{creative} if $A$ is RE and the complement of $A$ is productive. We denote the recursive Turing degree by $\mathbf{0}$, and the jump or completion of $\mathbf{0}$ by $\mathbf{0}^{\prime}$ (see \cite[p.256]{Rogers87}).

\begin{definition}[{\sf EU}, {\sf Creative}, and $\mathbf{0}^{\prime}$ theories]\label{def of 1.3}
Let $T$ be a consistent RE theory.
\begin{enumerate}[(1)]
    \item We say $T$ is \emph{essentially undecidable} ({\sf EU}) if any consistent RE extension of $T$ over the same language is undecidable.
\item We say $T$ is \emph{{\sf Creative}} if $T_P$ is creative.
\item We say $T$ is \emph{$\mathbf{0}^{\prime}$}  if $T$ has  Turing degree $\mathbf{0}^{\prime}$.
\end{enumerate}
\end{definition}

Now we introduce the notions of {\sf REW}, {\sf RFD}, {\sf RSS} and {\sf RSW} theories.
Given a consistent RE theory $T$, we denote the language of $T$ by $L(T)$.

\begin{definition}[{\sf REW}, {\sf RFD}, {\sf RSS} and {\sf RSW} theories]\label{property of theory}
Let $T$ be a consistent RE theory in the language including the constant symbol $\mathbf{0}$ and the function symbol $\mathbf{S}$. We define $\overline{n}=\mathbf{S}^n \mathbf{0}$ for $n \in \mathbb{N}$.
\begin{enumerate}
\item
We say an $n$-ary relation $R$ on $\mathbb{N}$ is \emph{weakly representable} in $T$ if there exists an $L(T)$-formula $\phi(x_1, \cdots, x_n)$ such that for any $\langle a_1, \cdots, a_n\rangle\in\mathbb{N}^n$, $R(a_1, \cdots, a_n)$ holds if and only if
$T\vdash \phi(\overline{a_1}, \cdots, \overline{a_n})$. We say  $\phi(x_1, \cdots, x_n)$ weakly represents the relation $\mathbf{R}$.

\item We say $T$ is {\sf REW} if any RE set is weakly representable in $T$.

  \item We say a total $n$-ary function $f$ on $\mathbb{N}$ is  \emph{definable} in $T$ if there exists an $L(T)$-formula $\phi(x_1, \cdots, x_n,y)$ such that for any $\langle a_1, \cdots, a_n\rangle\in\mathbb{N}^n$,
\[T\vdash \forall y [\phi(\overline{a_1}, \cdots, \overline{a_n},y)\leftrightarrow y= \overline{f(a_1, \cdots, a_n)}].\]
We say $\phi(x_1, \cdots, x_n,y)$ defines $f$ in $T$.
\item We say $T$ is  {\sf RFD} if every recursive function is definable in $T$.
 \item
We say an $n$-ary relation $R$ on $\mathbb{N}$ is \emph{strongly representable} in $T$ if there exists an $L(T)$-formula $\phi(x_1, \cdots, x_n)$ such that for any $\langle a_1, \cdots, a_n\rangle\in\mathbb{N}^n$, if $R(a_1, \cdots, a_n)$ holds, then
$T\vdash \phi(\overline{a_1}, \cdots, \overline{a_n})$, and if $R(a_1, \cdots, a_n)$ does not hold, then $T\vdash \neg\phi(\overline{a_1}, \cdots, \overline{a_n})$. We say  $\phi(x_1, \cdots, x_n)$ strongly represents the relation $R$.
  \item We say $T$ is  {\sf RSS} if any recursive set is strongly representable in $T$.
\item We say $T$ is  {\sf RSW} if any recursive set is weakly representable in $T$.
\end{enumerate}
\end{definition}

\begin{remark}\label{convention}
In this paper, we examine the following meta-mathematical properties of RE theories:
\begin{itemize}
  \item[{\sf Rosser}] {\sf Rosser} Theory: any disjoint  pair of RE sets is separable in the theory.
  \item[{\sf EI}] Effectively Inseparable:  the nuclei of the theory is $\sf EI$ (see Definition \ref{def of 1.1}).
  \item[{\sf RI}] Recursively Inseparable: the nuclei of the theory is $\sf RI$ (see Definition \ref{def of 1.2}).
  \item[{\sf TP}] Turing Persistent: see Definition \ref{def of 1.2}.
  \item[{\sf EHU}] Essentially Hereditarily Undecidable: any consistent RE extension of
it is hereditarily undecidable (i.e., its every sub-theory over the same
language is undecidable).
  \item[{\sf EU}] Essentially Undecidable: any consistent RE extension of the theory over the same language is undecidable.
  \item[{\sf Creative}] See Definition \ref{def of 1.3}.
  \item[$\mathbf{0}^{\prime}$] Having Turing degree $\mathbf{0}^{\prime}$.
  \item[{\sf REW}] Recursively Enumerable sets Weakly representable: any RE set is weakly representable in the theory.
  \item[{\sf RFD}] Recursive Functions Definable: Every recursive function is definable
in the theory.
  \item[{\sf RSS}] Recursive Sets Strongly representable: any recursive set is strongly representable in the theory.
  \item[{\sf RSW}] Recursive Sets Weakly representable: any recursive set is weakly representable in the theory.
\end{itemize}
\end{remark}

For the notion of {\sf Rosser} theories, see  \cite[p.221]{Smullyan}. For the notion of {\sf EI}  and  {\sf RI} theories, see \cite[p.119]{Smullyan}. For the notion of {\sf TP} theories, see \cite[p.8]{PV}. For the notion of {\sf HU} theories, see \cite[p.18]{Tarski}. For the notion of {\sf EU} theories, see \cite[p.14]{Tarski}. For the notion of {\sf Creative} theories, see \cite[p.119]{Smullyan}.
The notion of $\mathbf{0}^{\prime}$ is well known in computability theory and the notion of $\mathbf{0}^{\prime}$ theory is due to the author.
Conditions ``all RE sets are weakly representable", ``all recursive functions are definable", ``any recursive set is strongly representable" and ``any recursive set is weakly representable" appear in \cite[p.172]{Shoenfield61}. We isolated these conditions  and denoted  them  respectively by {\sf REW}, {\sf RFD}, {\sf RSS} and {\sf RSW}.

The motivation of this work is to systematically study the relationships between meta-mathematical properties of arithmetical theories in Remark \ref{convention}.  Given any two properties $P$ and $Q$ in Remark \ref{convention}, we examine whether the property $P$ implies $Q$. The structure of this paper is as follows. In Section 2, we introduce basic notions and facts that we will use throughout the paper. In Section 3, we examine the relationships between the theory $\mathbf{R}$ (see Definition \ref{def of R}) and all the properties  in Remark \ref{convention}. In Section 4, we examine the relationships between
{\sf Rosser} theories and the properties in Remark \ref{convention}. In Section 5, we examine the relationships between {\sf EI} theories as well as {\sf RI} theories, and the properties in Remark \ref{convention}. In Section 6, we examine the relationships between {\sf TP} theories as well as {\sf EHU} theories,  and the properties in Remark \ref{convention}. In Section 7, we examine the relationships between the remaining properties in Remark \ref{convention}.

\section{Basics}\label{section 2}

In this section, we examine some important properties of Robinson Arithmetic $\mathbf{Q}$ and the theory $\mathbf{R}$ that we will use throughout the paper. In Theorem \ref{relation about EI}, we examine some direct implicational relationships among properties in Remark \ref{convention}.\smallskip

\begin{definition}[Basic Notation]~\label{}
\begin{enumerate}[(1)]
 \item The function $\phi_e$ is the partial recursive function with
index $e$.
      \item  A $\Delta^0_0$ formula (\emph{bounded formula}) is built from atomic
 formulas using only negation, conjunction, disjunction, and bounded
quantifications (in the form $\exists x \leq y$ or $\forall x \leq y$). For $\Delta^0_0$ formulas, we can assume that the negation only appears before an atomic formula.
 \item A formula is  $\Sigma^0_1$ if it has the form $\exists x\phi$ where $\phi$ is $\Delta^0_0$.
\item Let $\mathfrak{N}=\langle\mathbb{N}, 0, S, +, \times\rangle$ be the standard model of arithmetic where $S, +$ and $\times$  are respectively  the successor function, the addition function  and  the multiplication function on $\mathbb{N}$.
\end{enumerate}
\end{definition}

The notion of interpretation provides us a method to compare different theories in different languages.
Let $T, S$ be consistent RE theories. Informally, we say $S$ \emph{interprets} $T$  (or $T$ is interpretable in $S$) if all sentences provable (refutable) in $T$ are mapped by a recursive translation function to sentences provable (refutable) in $S$. For the precise definition of \emph{interpretation}, we refer to \cite[pp.10-11]{on Q}.

\begin{remark}
Given two RE theories $S$ and $T$, we use $S\unlhd T$ to denote that $S$ is \emph{interpretable} in $T$ (see \cite[pp.10-11]{on Q}).
\end{remark}

Robinson Arithmetic $\mathbf{Q}$ and the theory $\mathbf{R}$ were introduced by Tarski, Mostowski and R.\ Robinson  in \cite[pp.51-53]{Tarski} which are important base theories in the study of incompleteness and undecidability, and have many nice meta-mathematical properties.\smallskip

\begin{definition}[Robinson Arithmetic $\mathbf{Q}$]~\label{def of Q}
Robinson Arithmetic $\mathbf{Q}$  is  defined in   the language $\{\mathbf{0}, \mathbf{S}, +, \times\}$ with the following axioms:
\begin{description}
  \item[$\mathbf{Q}_1$] $\forall x \forall y(\mathbf{S}x=\mathbf{S} y\rightarrow x=y)$;
  \item[$\mathbf{Q}_2$] $\forall x(\mathbf{S} x\neq \mathbf{0})$;
  \item[$\mathbf{Q}_3$] $\forall x(x\neq \mathbf{0}\rightarrow \exists y (x=\mathbf{S} y))$;
  \item[$\mathbf{Q}_4$]  $\forall x\forall y(x+ \mathbf{0}=x)$;
  \item[$\mathbf{Q}_5$] $\forall x\forall y(x+ \mathbf{S} y=\mathbf{S} (x+y))$;
  \item[$\mathbf{Q}_6$] $\forall x(x\times \mathbf{0}=\mathbf{0})$;
  \item[$\mathbf{Q}_7$] $\forall x\forall y(x\times \mathbf{S} y=x\times y +x)$.
\end{description}
\end{definition}

Given two theories $S$ and $T$, we say that $S$ is \emph{consistent with} $T$ if $S$ and $T$ has a common consistent extension over the same language.

\begin{fact}~\label{fact on Q}
\begin{enumerate}[(1)]
  \item Any theory consistent with $\mathbf{Q}$ over the same language  is undecidable (\cite[p.18]{Tarski}).
  \item The theory $\mathbf{Q}$ is minimal essentially undecidable: if an axiom of $\mathbf{Q}$ is deleted, then the remaining theory is not essentially undecidable (\cite[p.62]{Tarski}).
\end{enumerate}
\end{fact}

\begin{definition}\label{def of R}
Let $\mathbf{R}$ be the theory consisting of the following axiom schemes where $L(\mathbf{R})=\{\mathbf{0}, \mathbf{S}, +, \cdot, \leq\}$ and $x\leq y:=\exists z (z+x=y)$.
\begin{description}
  \item[\sf{Ax1}] $\overline{m}+\overline{n}=\overline{m+n}$;
  \item[\sf{Ax2}] $\overline{m}\cdot\overline{n}=\overline{m\cdot n}$;
  \item[\sf{Ax3}] $\overline{m}\neq\overline{n}$, if $m\neq n$;
  \item[\sf{Ax4}] $\forall x(x\leq \overline{n}\rightarrow x=\overline{0}\vee \cdots \vee x=\overline{n})$;
  \item[\sf{Ax5}] $\forall x(x\leq \overline{n}\vee \overline{n}\leq x)$.
\end{description}
\end{definition}

\begin{definition}\label{}
Suppose $T$ is a consistent RE theory in the language of $\mathbf{Q}$. We say $T$ is  \emph{$\Sigma^0_1$-complete} if any $\Sigma^0_1$ sentence true in the standard model $\mathfrak{N}$ is provable in $T$.
\end{definition}

In this paper, we use the following nice properties of $\mathbf{R}$:

\begin{fact}\label{key fact about R}~
\begin{enumerate}[(1)]
  \item All recursive  functions are definable in $\mathbf{R}$.
\item  The theory $\mathbf{R}$ is $\Sigma^0_1$-complete.
      \item The theory $\mathbf{R}$ is {\sf EHU}.
\end{enumerate}
\end{fact}
\begin{proof}\label{}
\begin{enumerate}[(1)]
  \item  This is proved by Tarski, Mostowski and R.\ Robinson  in \cite[p.56]{Tarski}.
  \item It is easy to prove the following two facts:
\begin{enumerate}[(A)]
  \item For any variable-free term $t$, there is a unique natural
number $n$ such that $\mathbf{R}\vdash t =\mathbf{S}^{n}\mathbf{0}$.
\item If $\tau$ is a quantifier-free sentence  true in $\mathfrak{N}$, then $\mathbf{R}\vdash \tau$.
\end{enumerate}
From these two facts, it is easy to show that  any $\Sigma^0_1$ sentence  true in $\mathfrak{N}$ is provable in $\mathbf{R}$.
\item Alan Cobham proved in 1957 that $\mathbf{R}$ is essentially hereditarily
undecidable. In 1962, Robert Vaught provided a new proof. See Visser's \cite[pp.2-5]{on Q}.
\end{enumerate}
\end{proof}

\begin{definition}\label{}
Suppose $\sigma(x)=\exists y\, \phi_0(x,y)$ and $\sigma^{\prime}(x)=\exists y\, \phi_1(x,y)$ are two $\Sigma^0_1$-formulas with only one free variable where $\phi_0(x,y)$ and $\phi_1(x,y)$ are  $\Delta^0_0$-formulas.
Define
\[\sigma(x)<\sigma^{\prime}(x):= \exists y(\phi_0(x,y)\wedge \forall z\leq y \neg\phi_1(x,z))\]
and \[\sigma^{\prime}(x)\leq \sigma(x):= \exists y (\phi_1(x,y)\wedge \forall z<y \neg\phi_0(x,z)).\]
\end{definition}
\smallskip

\begin{lemma}[\cite{CS}, Lemma 6.27, p.347]~\label{comparision lemma}
Suppose $\sigma(x)=\exists y\, \phi_0(x,y)$ and $\sigma^{\prime}(x)=\exists y\, \phi_1(x,y)$ are two $\Sigma^0_1$-formulas with only one free variable. If $\mathfrak{N}\models \sigma^{\prime}(\overline{n})\leq \sigma(\overline{n})$ for some $n\in\mathbb{N}$, then $\mathbf{R}\vdash \neg(\sigma(\overline{n})< \sigma^{\prime}(\overline{n}))$.
\end{lemma}
\begin{proof}\label{}
Suppose $\mathfrak{N}\models \sigma^{\prime}(\overline{n})\leq \sigma(\overline{n})$ for some $n\in\mathbb{N}$. For some $m\in\mathbb{N}$, we have $\mathfrak{N}\models \phi_1(\overline{n},\overline{m})\wedge \forall z<\overline{m} \neg\phi_0(\overline{n},z)$. By $\Sigma^0_1$-completeness of $\mathbf{R}$, we have $\mathbf{R}\vdash \phi_1(\overline{n},\overline{m})$ and $\mathbf{R}\vdash \neg\phi_0(\overline{n},\overline{i})$ for any $i<m$. Note that $\neg(\sigma(\overline{n})< \sigma^{\prime}(\overline{n}))\Leftrightarrow \forall y  (\phi_0(\overline{n},y)\rightarrow \exists z\leq y \phi_1(\overline{n},z))$. Since $\mathbf{R}\vdash \phi_1(\overline{n},\overline{m})$, we have $\mathbf{R}\vdash \overline{m}\leq y\rightarrow \exists z\leq y \phi_1(\overline{n},z)$. Since $\mathbf{R}\vdash y<\overline{m}\rightarrow y=\overline{0}\vee \cdots \vee y=\overline{m-1}$, we have $\mathbf{R}\vdash y<\overline{m}\rightarrow \neg\phi_0(\overline{n},y)$. Since $\mathbf{R}\vdash \overline{m}\leq y \vee y<\overline{m}$, we have $\mathbf{R}\vdash\neg(\sigma(\overline{n})< \sigma^{\prime}(\overline{n}))$.
\end{proof}

\begin{theorem}\label{R is Rosser}
Any disjoint pair of RE sets is separable in $\mathbf{R}$. Thus,  $\mathbf{R}$ is {\sf Rosser}.
\end{theorem}
\begin{proof}\label{}
Suppose $(A,B)$ is a disjoint pair of RE sets and $A,B$ are respectively definable in the standard model $\mathfrak{N}$ by the $\Sigma^0_1$ formulas $\sigma(x)$ and $\sigma^{\prime}(x)$. Suppose
$n\in A$. Then $\sigma(\overline{n})< \sigma^{\prime}(\overline{n})$ is a $\Sigma^0_1$ sentence true in $\mathfrak{N}$. By the  $\Sigma^0_1$-completeness of $\mathbf{R}$, we have $\mathbf{R}\vdash \sigma(\overline{n})< \sigma^{\prime}(\overline{n})$. Suppose $n\in B$. Then $\sigma^{\prime}(\overline{n})\leq \sigma(\overline{n})$ is true in $\mathfrak{N}$. By Lemma \ref{comparision lemma}, we have $\mathbf{R}\vdash \neg(\sigma(\overline{n})< \sigma^{\prime}(\overline{n}))$. Define $\psi(x):=\sigma(x)< \sigma^{\prime}(x)$. Then we have $n\in A\Rightarrow \mathbf{R}\vdash \psi(\overline{n})$ and $n\in B\Rightarrow \mathbf{R}\vdash \neg\psi(\overline{n})$. Thus, $\mathbf{R}$ is {\sf Rosser}.
\end{proof}
\smallskip

\begin{definition}[\cite{RM}, p.70]~\label{}
Let $(A,B)$ and $(C,D)$ be disjoint pairs of RE sets. We say $(A,B)$
is \emph{semi-reducible} to $(C,D)$ if there is a recursive function $f(x)$ such that if $x\in A$, then $f(x)\in C$, and if $x\in B$, then $f(x)\in D$.
\end{definition}
\smallskip

\begin{theorem}[pp.70-126 in \cite{Smullyan},  Theorem 2.10 in \cite{Cheng22}]~\label{EI thm}
For any consistent RE theory $T$, $T$ is {\sf EI} iff any disjoint pair $(A,B)$ of RE sets is semi-reducible to $(T_P, T_R)$.
\end{theorem}
\smallskip

\begin{fact}[\cite{Rogers87}, p.183, p.94]~\label{fact on creative}
\begin{enumerate}[(1)]
  \item Let $T$ be a consistent RE theory. The theory $T$ is {\sf Creative} iff any RE set is reducible to $T_P$: for any RE set $X$, there exists a recursive function $f$ such that $n\in X \Leftrightarrow f(n)\in T_P$.
  \item For any disjoint pair $(A,B)$ of RE sets, if $(A,B)$ is ${\sf EI}$, then both $A$ and $B$ are creative.
\end{enumerate}
\end{fact}

\begin{theorem}\label{relation about EI}
Let $T$ be a consistent RE theory.
\begin{enumerate}[(1)]
\item If  $T$ is {\sf Rosser}, then $T$ is ${\sf EI}$.
  \item If  $T$ is ${\sf EI}$, then $T$ is ${\sf RI}$.
  \item If $T$ is ${\sf RI}$, then $T$ is ${\sf EU}$.
  \item If $T$ is ${\sf EI}$, then $T$ is {\sf Creative}.
  \item If $T$ is {\sf Creative}, then $T$ is $\mathbf{0}^{\prime}$.
\end{enumerate}
\end{theorem}
\begin{proof}\label{}
\begin{enumerate}[(1)]
  \item If  $T$ is {\sf Rosser}, any disjoint pair of RE sets is semi-reducible to $(T_P, T_R)$. Thus, by Theorem \ref{EI thm}, $T$ is ${\sf EI}$.
  \item Suppose  $T$ is ${\sf EI}$, but it  is not ${\sf RI}$. Then there is a recursive set $S$ such that $T_P\subseteq S$ and $S\cap T_R=\emptyset$. Suppose $(T_P, T_R)$ is ${\sf EI}$ via the recursive function $f$, and $S=W_i$ and $\overline{S}$, the complement of $S$, is $W_j$. Since $T_P\subseteq W_i$ and $T_R\subseteq W_j$, we have $f(i,j)\notin W_i\cup W_j=\mathbb{N}$, which is a contradiction.
  \item Suppose  $T$ is ${\sf RI}$, but it  is not ${\sf EU}$. Let $S$ be a consistent RE extension of $T$ such that $S$ is recursive. Then $S_P$ is a recursive set separating $T_P$ and $T_R$ (i.e., $T_P\subseteq S_P$ and $S_P\cap T_R=\emptyset$), which contradicts the assumption that $T$ is ${\sf RI}$.
  \item  Follows from Fact \ref{fact on creative}(2).
  \item   Follows from Fact \ref{fact on creative}(1).
\end{enumerate}
\end{proof}

\section{Relationships with interpreting the theory $\mathbf{R}$}

In this section, we discuss the relationships between the theory $\mathbf{R}$ and the properties in Remark \ref{convention}. We will show in Theorem \ref{property of R} that the theory $\mathbf{R}$ has all these properties. In Theorem \ref{thm on R-like}, we will show that for any  property $P$  in Remark \ref{convention} and any consistent RE theory $T$, the assumption that $T$ has the property $P$  does not necessarily imply that $T$ interprets $\mathbf{R}$.

\begin{lemma}\label{RFD imply RSS}
\begin{enumerate}[(1)]
  \item The property {\sf RFD} implies {\sf RSS}.
  \item The property {\sf RSS} implies {\sf RSW}.
\end{enumerate}
\end{lemma}
\begin{proof}\label{}
\begin{enumerate}[(1)]
  \item Let $T$ be a consistent RE theory with {\sf RFD}. We show that $T$ is {\sf RSS}. Let $A$ be a recursive set and $c_A$ be the characteristic function of $A$. Since $c_A$ is recursive and $T$ is {\sf RFD}, let $\phi(x,y)$ be the representation formula of $c_A$ in $T$. Define $\psi(x):=\phi(x,\overline{1})$. Then $\psi(x)$ strongly represents $A$ in $T$.
  \item Follows from the definition.
\end{enumerate}
\end{proof}

\begin{theorem}\label{property of R}
The theory $\mathbf{R}$ has all the properties in Remark \ref{convention}.
\end{theorem}
\begin{proof}\label{}
\begin{enumerate}[(1)]
\item  From Theorem \ref{R is Rosser} and Theorem \ref{relation about EI}, $\mathbf{R}$ is {\sf Rosser}, {\sf EI}, {\sf RI}, {\sf EU},  {\sf Creative} and $\mathbf{0}^{\prime}$.
  \item Since any consistent RE extension of $\mathbf{R}$ has  Turing degree $\mathbf{0}^{\prime}$, $\mathbf{R}$ is {\sf TP}.
  \item From Fact \ref{key fact about R}(3), $\mathbf{R}$ is {\sf EHU}.
  \item Let $A$ be an RE set. The set $A$ is definable in the standard model $\mathfrak{N}$ by a $\Sigma^0_1$-formula. I.e., there is a $\Sigma^0_1$-formula $\phi(x)$ such that $n\in A\Leftrightarrow \mathfrak{N}\models \phi(\overline{n})$. Since $\mathbf{R}$ is $\Sigma^0_1$-complete, we have $\mathfrak{N}\models \phi(\overline{n})\Leftrightarrow \mathbf{R}\vdash \phi(\overline{n})$. Thus, $A$ is weakly representable in $\mathbf{R}$. Thus, $\mathbf{R}$ is {\sf REW}.
\item From Fact \ref{key fact about R}(1), any recursive function is definable in $\mathbf{R}$. Thus, $\mathbf{R}$ is {\sf RFD}.
\item  From (5) and Lemma \ref{RFD imply RSS}, $\mathbf{R}$ is {\sf RSS} and {\sf RSW}.
\end{enumerate}
\end{proof}

Now, we show that {\sf Rosser}  does not imply interpreting $\mathbf{R}$.
\smallskip

\begin{theorem}[\cite{EHU}, Theorem 3.1-3.2]~\label{EHU CT}
Let $T$ be a consistent RE theory. The following are equivalent:
\begin{enumerate}[(1)]
  \item The theory $T$ is  {\sf EHU}.
  \item For any consistent RE theory $S$ over the language of $T$, if $S$ has a consistent RE extension  $W$ such that $T$ is interpretable in $W$, then $S$ is undecidable.
  \item For any RE theory $S$ over the  language of $T$, if $S$ is consistent with $T$, then $S$ is undecidable.
\end{enumerate}
\end{theorem}

\begin{proposition}\label{coro of R-like}
Let $T$ be a consistent RE theory. If $T$ interprets $\mathbf{R}$, then $T$ has the following properties:  {\sf EI}, {\sf RI}, {\sf TP}, {\sf EHU}, {\sf EU},  {\sf Creative} and $\mathbf{0}^{\prime}$.
\end{proposition}
\begin{proof}\label{}
We only show that $T$ is {\sf EI}, {\sf TP} and {\sf EHU}. The other claims follow from Theorem \ref{relation about EI}.

We first show that $T$ is {\sf EI}. Let $(A,B)$ be a disjoint pair of RE sets. By Theorem \ref{EI thm}, it suffices to show that $(A,B)$ is semi-reducible to $(T_P, T_R)$. Since $\mathbf{R}$ is {\sf EI}, by Theorem \ref{EI thm}, there exists a recursive function $f$ such that $n \in A\Rightarrow f(n)\in \mathbf{R}_P$ and $n \in B\Rightarrow f(n)\in \mathbf{R}_R$. Let $I$ be the recursive translation function such that for any sentence $\phi$, if $\mathbf{R}\vdash\phi$, then $T\vdash \phi^{I}$.

Define the function $g$ as follows:
\[
 g(n) =
  \begin{cases}
   \ulcorner \phi^{I}\urcorner, & \text{if $n=\ulcorner \phi\urcorner$ for some sentence $\phi$;}\\
   0,       & \text{otherwise.}
  \end{cases}
\]
Define $h(n)=g(f(n))$. Note that $h$ is recursive since $f, g$ are recursive. Suppose $n \in A$ and $f(n)=\ulcorner\phi\urcorner$ such that $\mathbf{R}\vdash\phi$. Then $h(n)\in T_P$ since $T\vdash \phi^{I}$ and $h(n)=\ulcorner \phi^{I}\urcorner$. Suppose $n \in B$ and $f(n)=\ulcorner\theta\urcorner$ such that $\mathbf{R}\vdash\neg\theta$. Then $h(n)\in T_R$ since $T\vdash \neg\theta^{I}$ and $h(n)=\ulcorner \theta^{I}\urcorner$. Thus, $(A,B)$ is semi-reducible to $(T_P, T_R)$ via the recursive function $h$.

Suppose $\mathbf{R}\unlhd T$ and $S$ is a consistent RE extension of $T$. Then $S$ has  Turing degree $\mathbf{0}^{\prime}$. Hence $T$ is {\sf TP}.

Suppose $\mathbf{R}\unlhd T, S$ is a consistent RE extension of $T$ over the same language  and $W$ is a sub-theory of $S$ over the same language. Since $\mathbf{R}$ is {\sf EHU} and $W$ has a consistent RE extension interpreting $\mathbf{R}$, by Theorem \ref{EHU CT}, $W$ is undecidable. Thus, $T$ is {\sf EHU}.
\end{proof}
\smallskip

\begin{fact}[\cite{RM}, p.57]~\label{special rf}
There exists a recursive function $e(x, y)$ such that for any $i, j\in\omega$, $W_{e(i, j)}=\{x: \exists y[T_1(i, x, y) \wedge \forall z\leq y \neg T_1(j, x, z)]\}$. The function $e(x, y)$ has the following property: if $W_i\cap W_j=\emptyset$, then $e(i,j)=i$ and $e(j,i)=j$.
\end{fact}

\begin{example}\label{Putnam E}
Hilary Putnam  constructs in \cite[p.53]{Putnam 60} a decidable theory ${\sf D}$ with infinite
signature, having an essentially undecidable extension ${\sf E}$ as follows:
${\sf D}$ is monadic quantification theory with an infinite
set of individual constants $\overline{1}, \overline{2}, \cdots, \overline{n}, \cdots$ and an infinite set of monadic predicate $P_1, P_2, \cdots, P_n, \cdots$.
The theory ${\sf E}$  is obtained by adjoining to ${\sf D}$ the following infinite list of axioms where $e(x, y)$ is the recursive function in Fact \ref{special rf}:
\begin{description}
  \item[$A_1$] $P_i(\overline{n})$, if $n \in W_i$;
  \item[$A_2$] $\forall x [P_{e(i,j)}(x)\rightarrow \neg P_{e(j,i)}(x)]$.
\end{description}
Putnam shows in \cite[pp.53-54]{Putnam 60} that ${\sf E}$ is {\sf EU}. Smullyan shows in \cite[Theorem 3]{Smullyan} that  ${\sf E}$ is {\sf Rosser}.
\end{example}

\begin{theorem}\label{R-like imply HU}
Let $T$ be a consistent RE theory. If $T$ interprets $\mathbf{R}$, then $T$ is {\sf HU}.
\end{theorem}
\begin{proof}\label{}
Suppose $T$ interprets $\mathbf{R}$ but $T$ is not {\sf HU}.
Then there exists a sub-theory $S$ of $T$ over the same language  such that $S$ is decidable.
Since $\mathbf{R}$ is  {\sf EHU} and $S$ has an extension interpreting $\mathbf{R}$, by Theorem \ref{EHU CT}, $S$ is undecidable which is a contradiction.
\end{proof}

\begin{theorem}\label{Rosser not imply R-like}
For a consistent RE theory $T$, in general
``$T$ is {\sf Rosser}" does not imply ``$T$ interprets $\mathbf{R}$".
\end{theorem}
\begin{proof}\label{}
Consider Putnam's theory ${\sf E}$ in Example \ref{Putnam E}. The theory ${\sf E}$ is {\sf Rosser} from Example \ref{Putnam E}.
Since ${\sf D}$ is a decidable sub-theory of ${\sf E}$ over the same language, ${\sf E}$ is not {\sf HU}. By Theorem \ref{R-like imply HU}, ${\sf E}$ does not interpret $\mathbf{R}$.
\end{proof}

Now, we show that ${\sf EI}$  does not imply interpreting $\mathbf{R}$.

\begin{remark}\label{theory of J}
In this paper, we use Janiczak's theory {\sf J} introduced in \cite[p.136]{Janiczak}, which is a theory  having only one binary relation symbol $E$ with the
following axioms.
\begin{description}
  \item[{\sf J1}] $E$ is an equivalence relation.
  \item[{\sf J2}] For any $n\in\omega$,  there is at most one equivalence class of size precisely $n$.
  \item[{\sf J3}] For any $n\in\omega$, there are at least $n$ equivalence classes with at least $n$ elements.\footnote{We include the axiom {\sf J3} to make the proof of the following fact in Theorem \ref{J thm} easier: over {\sf J}, every sentence is equivalent to a Boolean combination of $A_n$'s.}
\end{description}
We define $A_n$ to be the sentence claiming that there exists an equivalence class of size precisely $n + 1$. Note that $A_n$'s are mutually independent over {\sf J}.
\end{remark}
\smallskip

\begin{theorem}~\label{J thm}
\begin{enumerate}[(1)]
  \item {\sf J} is decidable. (see \cite[Theorem 4]{Janiczak})
  \item Over {\sf J}, every sentence is equivalent with a Boolean combination of $A_n$'s.  In fact, given a sentence $\phi$, we can effectively find a Boolean combination of $A_n$'s which is equivalent with $\phi$ over {\sf J}. (see  \cite[Lemma 2]{Janiczak})
\end{enumerate}
\end{theorem}
\smallskip

\begin{theorem}[The $s$-$m$-$n$ theorem, \cite{Rogers87}, p.23]~\label{}
For any $m, n\geq 1$, there exists a recursive function $s_n^m$ of $m+1$ variables such that for all $x, y_1, \cdots, y_m, z_1, \cdots, z_n$, we have
\[\phi^n_{s_n^m(x, y_1, \cdots, y_m)}(z_1, \cdots, z_n)=\phi_{x}^{m+n}(y_1, \cdots, y_m, z_1, \cdots, z_n).\]
\end{theorem}

\begin{theorem}\label{EI not R}
Define the theory $T:= {\sf J} + \{A_n: n\in B\} +\{\neg A_n: n\in C\}$ where $(B,C)$
is a disjoint pair of RE sets.
\begin{enumerate}[(1)]
  \item The theory $T$  does not interpret $\mathbf{R}$.
  \item If $(B,C)$ is ${\sf RI}$, then $T$ is ${\sf RI}$.
  \item If $(B,C)$ is ${\sf EI}$, then $T$ is ${\sf EI}$.
\end{enumerate}
\end{theorem}
\begin{proof}\label{}
\begin{enumerate}[(1)]
  \item
Since {\sf J} is decidable, $T$ is not {\sf HU}. By Theorem \ref{R-like imply HU},  $T$  does not interpret $\mathbf{R}$.
\item Suppose $(B,C)$ is ${\sf RI}$ and $T$ is not {\sf RI}. Then there exists a recursive set $X$ such that $T_P\subseteq X$ and $X\cap T_R=\emptyset$. Define the function $f: n\mapsto \ulcorner A_n\urcorner$. Let $Y:=f^{-1}[X]$. Then $Y$ is a recursive set separating $B$ and $C$ since $f$ is recursive, $B\subseteq Y$ and $Y\cap C=\emptyset$. Thus, $(B,C)$ is not ${\sf RI}$ which is a contradiction.
\item We want to find a recursive function $h(i,j)$ such that if $T_P\subseteq W_i, T_R\subseteq W_j$ and $W_i\cap W_j=\emptyset$, then $h(i,j)\notin W_i\cup W_j$.
Define the function $f: n\mapsto \ulcorner A_n\urcorner$. Note that $f$ is recursive. By the s-m-n theorem, there is a recursive function $g$ such that $f^{-1}[W_i]=W_{g(i)}$. Suppose $(B,C)$ is {\sf EI} via the recursive function $t(i,j)$. Define $h(i,j)=f(t(g(i), g(j)))$. Note that $h$ is recursive.

Suppose
$T_P\subseteq W_i, T_R\subseteq W_j$ and $W_i\cap W_j=\emptyset$.
Note that $B\subseteq f^{-1}[T_P]\subseteq f^{-1}[W_i]=W_{g(i)}$, and $C\subseteq f^{-1}[T_R]\subseteq f^{-1}[W_j]=W_{g(j)}$.
Note that $W_{g(i)}\cap W_{g(j)}=\emptyset$ since $W_i\cap W_j=\emptyset$.
Since $(B,C)$ is {\sf EI} via the function $t$, we have $t(g(i), g(j))\notin W_{g(i)}\cup W_{g(j)}$. Then, $h(i,j)\notin W_i\cup W_j$. Thus, $T$ is {\sf EI}.
\end{enumerate}
\end{proof}

Now, we show that ${\sf TP}$  does not imply interpreting $\mathbf{R}$.\smallskip

\begin{theorem}[\cite{Shoenfield 58}, Theorem 1]~\label{Shoenfield first}
For any non-recursive RE set $A$, we can effectively find an {\sf RI} pair $(B,C)$ such  that $B,C {\sf \leq_T} A$ and for any RE set $D$ separating $B$ and $C$ (i.e., $B\subseteq D$ and $D\cap C=\emptyset$), we have $A {\sf \leq_T} D$.
\end{theorem}

\begin{theorem}\label{thm on TP original}
For any non-recursive RE Turing degree $\mathbf{d}$, we can effectively find a consistent RE theory $T_{\mathbf{d}}$ such that $T_{\mathbf{d}}$ does not interpret $\mathbf{R}$, $T_{\mathbf{d}}$ is {\sf RI}, {\sf TP}, and has Turing degree $\mathbf{d}$.
\end{theorem}
\begin{proof}\label{}
Suppose $A$ is a non-recursive RE set with Turing degree $\mathbf{d}$. From Theorem \ref{Shoenfield first},  we can effectively find an {\sf RI} pair $(B,C)$ such  that $B,C {\sf \leq_T} A$ and for any RE set $D$ separating $B$ and $C$, we have $A{\sf \leq_T} D$.
Define the theory $T_{\mathbf{d}}:= {\sf J} + \{A_n: n\in B\} +\{\neg A_n: n\in C\}$.
\begin{enumerate}[(1)]
  \item Since $T_{\mathbf{d}}$ has a decidable sub-theory {\sf J}, by Theorem \ref{R-like imply HU}, $T_{\mathbf{d}}$ does not interpret $\mathbf{R}$.
  \item From Theorem \ref{EI not R}(2), $T_{\mathbf{d}}$ is {\sf RI} since $(B,C)$ is {\sf RI}.
  \item Note that $B$ is recursive in $T_{\mathbf{d}}$. Since $B$ is an RE set  separating $B$ and $C$, we have $A {\sf \leq_T} B$. Thus, $A$ is recursive in $T_{\mathbf{d}}$.
From Theorem \ref{J thm}(2), $T_{\mathbf{d}}$ is recursive in $B,C$. Since $B,C {\sf \leq_T} A$, the theory $T_{\mathbf{d}}$ is recursive in $A$. Thus, $T_{\mathbf{d}}$ has Turing degree $\mathbf{d}$.
  \item
Let $S$ be a consistent RE extension of $T_{\mathbf{d}}$. We show that $T_{\mathbf{d}} {\sf \leq_T} S$. Define $X=\{n: S\vdash A_n\}$. Note that $B\subseteq X$ and $X\cap C=\emptyset$. From Theorem \ref{Shoenfield first}, we have $A {\sf \leq_T} X$. Since $X {\sf \leq_T} S$ and $A$ has the same Turing degree as $T_{\mathbf{d}}$,  we have $T_{\mathbf{d}} {\sf \leq_T} S$. Thus, $T_{\mathbf{d}}$ is {\sf TP}.
\end{enumerate}
\end{proof}

Now, we show that {\sf EHU}  does not imply interpreting $\mathbf{R}$.

\begin{definition}\label{}
We say that two RE theories $S$ and $T$ are \emph{Boolean recursively isomorphic}  if there is a bijective recursive function $\Phi$ from  sentences in the language of $S$ to sentences in the language of $T$ such that:
\begin{enumerate}[(1)]
  \item $\Phi$ commutes with the boolean connectives, i.e., $\Phi(\neg\phi) = \neg \Phi(\phi)$ and $\Phi(\phi\rightarrow\psi) =\Phi(\phi)\rightarrow \Phi(\psi)$;
  \item  $S \vdash\phi \Leftrightarrow T\vdash \Phi(\phi)$.
\end{enumerate}
We say $\Phi$ is a Boolean recursive isomorphism between $S$ and $T$.
\end{definition}
\smallskip

\begin{theorem}[\cite{FAT}, Theorem 7.1.3]~\label{RBM key thm}
Suppose $T$ is an RE theory with index $i$. Then, there
is a finitely axiomatised theory $S := pere(i)$ such that there is a Boolean recursive isomorphism $\Phi$ between $T$ and $S$.
\end{theorem}
\smallskip

\begin{theorem}[\cite{Shoenfield 58}, Theorem 2]~\label{Shoenfield EU}
For any non-recursive  RE Turing degree $\mathbf{d}$, there exists an essentially undecidable theory with the degree $\mathbf{d}$.
\end{theorem}
\smallskip

\begin{theorem}[\cite{Tarski}, Chapter I, Theorem 6]~\label{Thm on EHU}
Let $T$ be a consistent RE theory.   If $T$ is finitely axiomatizable and essentially undecidable, then $T$ is essentially hereditarily undecidable.
\end{theorem}
\smallskip

\begin{theorem}[\cite{Hanf}, Theorem 3.3]~\label{EHU degree}
Let $\mathbf{d}$ be any non-recursive RE Turing degree. Then there is a
finitely axiomatised essentially hereditarily undecidable theory with Turing degree $\mathbf{d}$.
\end{theorem}
\begin{proof}\label{}
From Theorem \ref{Shoenfield EU}, there exists an {\sf EU} theory $T$ with Turing degree $\mathbf{d}$. Suppose $T$ has index $i$. From Theorem \ref{RBM key thm}, there
is a finitely axiomatised theory $S := pere(i)$ such that $T$ is Boolean recursively
isomorphic to $S$. Since Boolean recursive  isomorphisms preserve essential undecidability, $S$ is finitely axiomatizable and essentially undecidable. Thus, $S$ is {\sf EHU} by Theorem \ref{Thm on EHU}. Since Boolean recursive  isomorphisms preserve Turing degrees,  $S$ has Turing degree $\mathbf{d}$.
\end{proof}

Now, we show that ${\sf REW}$  does not imply interpreting $\mathbf{R}$.

\begin{theorem}\label{REW not imply R-like}
The property ${\sf REW}$ does not imply {\sf HU}.
\end{theorem}
\begin{proof}\label{}
We first show that Putnam's theory ${\sf E}$ in Example \ref{Putnam E} is {\sf REW}. Let $A=W_i$ be an RE set. From Example \ref{Putnam E}, ${\sf E}\vdash P_i(\overline{n})$ if $n\in A$. We show that if ${\sf E}\vdash P_i(\overline{n})$, then $n\in A$. Suppose ${\sf E}\vdash P_i(\overline{n})$, but $n\notin A$.  Suppose $\{n\}=W_j$. Take the recursive function $e(i,j)$ as in Fact \ref{special rf}.
Since $W_i\cap W_j=\emptyset$, by Fact \ref{special rf}, $e(i,j)=i$ and $e(j,i)=j$. From Example \ref{Putnam E}, we have ${\sf E}\vdash P_i(\overline{n})\rightarrow \neg P_j(\overline{n})$. Thus, ${\sf E}\vdash \neg P_j(\overline{n})$. Since $n\in W_j$, we have ${\sf E}\vdash P_j(\overline{n})$, which is a contradiction. But ${\sf E}$ is not {\sf HU} since {\sf E} has a decidable sub-theory.
\end{proof}

The following theorem is an important tool for examining the relationships among the  properties  in Remark \ref{convention}.\smallskip

\begin{theorem}[\cite{Shoenfield61}, pp.172-173]~\label{Shoefield theory}
There is a theory $T$ in which any recursive function is definable but $T$ is not {\sf Creative}, and no non-recursive set is weakly representable. In fact, $T$ has Turing degree \emph{less than}  $\mathbf{0}^{\prime}$.
\end{theorem}

The following theorem is a summary of the relationships between interpreting $\mathbf{R}$ and the properties  in Remark \ref{convention}.

\begin{theorem}\label{thm on R-like}
None of the properties in Remark \ref{convention} can imply interpreting $\mathbf{R}$: for any   property $P$ in Remark \ref{convention} and  a consistent RE theory $T$,  in general ``$T$ has the property $P$"  does not necessarily imply ``$T$ interprets $\mathbf{R}$".
\end{theorem}
\begin{proof}\label{}
\begin{enumerate}[(1)]
  \item From Theorem \ref{Rosser not imply R-like}, {\sf Rosser}  does not imply interpreting $\mathbf{R}$.
\item From Theorem \ref{EI not R}, ${\sf EI}$ does not imply interpreting $\mathbf{R}$.
  \item From (2) and ${\sf EI}\Rightarrow {\sf RI}$, ${\sf RI}$ does not imply interpreting $\mathbf{R}$.
  \item From Theorem \ref{thm on TP original}, ${\sf TP}$ does not imply interpreting $\mathbf{R}$.
      \item From Proposition \ref{coro of R-like}, theories  interpreting $\mathbf{R}$ have  Turing degree $\mathbf{0}^{\prime}$. From Theorem \ref{EHU degree}, {\sf EHU} theories can have any  non-zero RE Turing degree. Thus, {\sf EHU} does not imply interpreting $\mathbf{R}$.
  \item From (2) and ${\sf EI}\Rightarrow {\sf EU}$, ${\sf EU}$ does not imply interpreting $\mathbf{R}$.
  \item From (2) and ${\sf EI} \Rightarrow$ {\sf Creative}, {\sf Creative}   does not imply interpreting $\mathbf{R}$.
  \item From (2) and {\sf EI} implies $\mathbf{0}^{\prime}$, $\mathbf{0}^{\prime}$   does not imply interpreting $\mathbf{R}$.
   \item From Theorem \ref{REW not imply R-like}, ${\sf REW}$ does not imply {\sf HU}. From Theorem \ref{R-like imply HU}, interpreting $\mathbf{R}$ implies {\sf HU}. Thus, ${\sf REW}$ does not imply interpreting $\mathbf{R}$.
   \item From Theorem \ref{Shoefield theory} and the fact that theories  interpreting $\mathbf{R}$ are {\sf Creative}, none of {\sf RFD}, {\sf RSS} and {\sf RSW} implies  interpreting $\mathbf{R}$.
\end{enumerate}
\end{proof}

\section{Relationships with Rosser theories}

In this section, we discuss the relationships between Rosser theories and the properties in Remark \ref{convention}. We first show that the property {\sf EI} does not imply {\sf RSW}.\smallskip

\begin{definition}[Two theories: {\sf Succ} and ${\sf Succ^{-}}$]~\label{}
\begin{enumerate}[(1)]
  \item The theory of successor, {\sf Succ},  is  defined in   the language $\{\mathbf{0}, \mathbf{S}\}$ consisting of the following axioms:
\begin{description}
  \item[S1] $\forall x \forall y(\mathbf{S}x=\mathbf{S} y\rightarrow x=y)$;
  \item[S2] $\forall x(\mathbf{S} x\neq \mathbf{0})$;
  \item[S3] $\forall x(x\neq \mathbf{0}\rightarrow \exists y (x=\mathbf{S} y))$;
  \item[S4.n]  $\forall x (\mathbf{S}^{n}x \neq x)$ for each $n\in\mathbb{N}$.
\end{description}
  \item Let ${\sf Succ^{-}}$ be the sub-theory of {\sf Succ} over the same language consisting of axioms $\mathbf{S1}, \mathbf{S2}$ and $\mathbf{S3}$.
\end{enumerate}
\end{definition}

\begin{remark}
Any model of {\sf Succ} consists of the standard part (isomorphic to $\mathbb{N}$) plus some number of $\mathbb{Z}$-chains (including the case of no number of $\mathbb{Z}$-chain at all). Any model of ${\sf Succ^{-}}$ consists of the standard part (isomorphic to $\mathbb{N}$), some number  of $\mathbb{Z}$-chains (including the case of no number of $\mathbb{Z}$-chain at all), and some number of cycles with length $n$ for some $n\in \omega$ (including the case of no cycle with length $n$ at all), where we say that a cycle with length $n$ is of the form: $\mathbf{S}x_i=x_{i+1}$ for $0\leq i<n$, and $\mathbf{S}x_n=x_0$.
\end{remark}
\smallskip

\begin{fact}[\cite{Enderton}, p.190]~\label{fact on Succ}
\begin{enumerate}[(1)]
  \item The theory $\sf Succ$ is decidable.
  \item The theory $\sf Succ$ is $\kappa$-categorical for uncountable $\kappa$: all models of $\sf Succ$ with cardinality $\kappa$ are isomorphic.
\end{enumerate}
\end{fact}

\begin{theorem}\label{thm on Succ}
Let $S$ be a consistent extension of ${\sf Succ^{-}}$ over the same language. Then for any $X\subseteq \mathbb{N}$, $X$ is weakly representable in $S$ iff $X$ is finite or co-finite.
\end{theorem}
\begin{proof}\label{}
It is easy to show that if $X$ is finite or co-finite, then $X$ is weakly representable in $S$. For example, if $X^{c}$, the complement of $X$, is $\{a_1, \cdots, a_n\}$, then $X$ is weakly representable in $S$ by the formula $\phi(x):= x\neq \overline{a_1}\wedge\cdots\wedge x\neq \overline{a_n}$.

Now we show that if $X$ is weakly representable in $S$, then $X$ is finite or co-finite. Suppose $X$ is weakly representable in $S$  via a formula $\phi(x)$, but $X$ is neither finite nor co-finite.

Note that $X$ is finite iff $\exists n\forall m>n (m\notin X)$, and $X$ is co-finite iff $\exists n\forall m>n (m\in X)$. Thus, if $X$ is neither finite nor co-finite, then for every $n$, there exists $m>n$ such that $m\in X$ and there exists $m>n$ such that $m\notin X$. Hence, for any $n$, there exists $m>n$ such that $m\in X$ and $m+1\notin X$.

We expand the language of $S$ by adding a new constant $c$. Define the theory $T$ in this new language as follows. Let $T:= S+\phi(c)+\neg \phi(\mathbf{S}c)+\{c\neq \overline{n}: n\in\omega\}$ where $\phi(x)$ is the formula which weakly represents $X$ in $S$.
\begin{lemma}
The theory $T$ is consistent.
\end{lemma}
\begin{proof}\label{}
We show that any finite sub-theory $W$ of $T$ has a model. Let $W= U+\phi(c)+\neg \phi(\mathbf{S}c)+\{c\neq \overline{n}: n\in\{a_1, \cdots, a_k\}\}$ where $U$ is a finite sub-theory of $S$. Let $n=\max(a_1, \cdots, a_k)$. Then there exists $m>n$ such that $m\in X$ and $m+1\notin X$. Since $n\in X \Leftrightarrow S\vdash \phi(\overline{n})$, we have $S\nvdash \phi(\mathbf{S}\overline{m})$. Thus, there exists a model $\mathfrak{M}$ in the language of $S$ such that $\mathfrak{M}\models S$ and $\mathfrak{M}\models \neg\phi(\mathbf{S}\overline{m})$. Since $m\in X$, $\mathfrak{M}\models \phi(\overline{m})$.
Let $\mathfrak{M}^{\prime}$ be the expansion of $\mathfrak{M}$ in the language of $T$ such that
$c^{\mathfrak{M}^{\prime}}=\overline{m}^{\mathfrak{M}}$. Then $\mathfrak{M}^{\prime}\models W$. By the compactness theorem, $T$ is consistent.
\end{proof}

Suppose $\mathfrak{M}^{\ast}\models T$. Let $c^{\mathfrak{M}^{\ast}}=a$. Then $a$ is either in a $\mathbb{Z}$-chain $C$ or in a cycle $C$. Define a function $f$ on the domain of $\mathfrak{M}^{\ast}$ as:
\[
 f(x) =
  \begin{cases}
   \mathbf{S}^{\mathfrak{M}^{\ast}}x, & \text{if $x\in C$;} \\
   x,       & \text{otherwise.}
  \end{cases}
\]
Note that $f$ is an automorphism from $\mathfrak{M}^{\ast}$ to $\mathfrak{M}^{\ast}$ in the language of $S$. Thus, we have for any $b$ in the domain of $\mathfrak{M}^{\ast}, \mathfrak{M}^{\ast}\models \phi[b]\Leftrightarrow \mathfrak{M}^{\ast}\models \phi[f(b)]$. Since $\mathfrak{M}^{\ast}\models \phi(c)$, we have $\mathfrak{M}^{\ast}\models \phi[a]$ and thus $\mathfrak{M}^{\ast}\models \phi[f(a)]$. Since $\mathfrak{M}^{\ast}\models \phi[f(a)]$ and $(\mathbf{S}c)^{\mathfrak{M}^{\ast}}=\mathbf{S}^{\mathfrak{M}^{\ast}}a=f(a)$, thus $\mathfrak{M}^{\ast}\models \phi(\mathbf{S}c)$ which contradicts  $\mathfrak{M}^{\ast}\models \neg \phi(\mathbf{S}c)$.

Thus, we have proved that $X$ is weakly representable in $S$ iff $X$ is finite or co-finite.
\end{proof}

\begin{lemma}\label{Rosser implies RSS}
\begin{enumerate}[(1)]
  \item The property {\sf REW} implies {\sf RSW}.
  \item {\sf Rosser} implies {\sf RSS}.
\end{enumerate}
\end{lemma}
\begin{proof}\label{}
\begin{enumerate}[(1)]
  \item Follows from the definition.
  \item Suppose $T$ is {\sf Rosser}. Let $A$ be a recursive set. Then $(A,\overline{A})$ is a disjoint pair of RE sets. Since $T$ is {\sf Rosser}, there exists a formula $\phi(x)$ such that $n\in A\Rightarrow T\vdash \phi(\overline{n})$ and $n\in \overline{A}\Rightarrow T\vdash \neg\phi(\overline{n})$. Thus $A$ is strongly representable in $T$. Hence $T$ is {\sf RSS}.
\end{enumerate}
\end{proof}

The following theorem is a corollary of Theorem \ref{thm on Succ}.\smallskip

\begin{theorem}~\label{EI does not imply REW}
\begin{enumerate}[(1)]
\item The property {\sf EI} does not imply {\sf Rosser}.
  \item The property {\sf EI} does not imply {\sf REW}.
  \item The property {\sf EI} does not imply {\sf RFD}.
  \item The property {\sf EI} does not imply {\sf RSS}.
  \item The property {\sf EI} does not imply {\sf RSW}.
\end{enumerate}
\end{theorem}
\begin{proof}\label{}
We work in the language of {\sf Succ}. Define the sentence $\phi_n:= \exists x (\mathbf{S}^{n}x=x)$.
Define the  theory $T:= {\sf Succ^{-}} + \{\phi_n: n\in B\} + \{\neg \phi_n: n\in C\}$ where $(B,C)$ is an {\sf EI} pair. Note that $T$ is {\sf EI}.

From Lemma \ref{RFD imply RSS} and Lemma \ref{Rosser implies RSS}, it suffices to show that $T$ is not {\sf RSW}.
Note that $T$ is a consistent extension of ${\sf Succ^{-}}$. From Theorem \ref{thm on Succ}, the weakly representable sets in $T$ are exactly the
finite or cofinite sets. Thus, $T$ is not {\sf RSW}.
\end{proof}

\begin{definition}\label{}
Let $T$ and $S$ be consistent RE theories. We say Boolean recursive isomorphisms \emph{preserve} some property $P$ if the following holds: if $T$ has the property $P$ and $T$ is Boolean recursively isomorphic to $S$, then $S$ also has the property $P$.
\end{definition}

As an application of Theorem \ref{EI does not imply REW}, we answer the following question: among the properties in Remark \ref{convention}, which ones are preserved under Boolean recursive
isomorphisms, and which ones are not preserved under Boolean recursive
isomorphisms.\smallskip

\begin{theorem}[\cite{Pour-EI-Kripke}, Theorem 2]~\label{EI BRI}
All {\sf EI} theories are Boolean recursively isomorphic.
\end{theorem}
\smallskip

\begin{theorem}[\cite{EHU}, Theorem 2.13]~\label{BRI HU}
Any RE theory is Boolean recursively isomorphic to an RE theory which is not {\sf EHU}.
\end{theorem}
\smallskip

\begin{theorem}~\label{}
\begin{enumerate}[(1)]
  \item Boolean recursive isomorphisms preserve the following properties: {\sf EI}, {\sf RI}, {\sf TP}, {\sf EU}, {\sf Creative} and $\mathbf{0}^{\prime}$.
  \item Boolean recursive isomorphisms do not preserve the following properties: interpreting $\mathbf{R}$, {\sf EHU}, {\sf Rosser}, {\sf REW}, {\sf RFD}, {\sf RSS} and {\sf RSW}.
\end{enumerate}
\end{theorem}
\begin{proof}\label{}
(1): We only show that Boolean recursive isomorphisms preserve {\sf EI} and {\sf TP} theories. Others are easy to check from definitions.

Let $S$ and $T$ be RE theories. Suppose $S$ is {\sf EI} and $S$ is Boolean recursively isomorphic to $T$ via a recursive bijection $g$. We show that $T$ is {\sf EI}. Let $(A,B)$ be a disjoint pair  of RE sets. Since $S$ is {\sf EI},
by Theorem \ref{EI thm}, there exists a recursive function $f$ such that $n\in A\Rightarrow f(n)\in S_P$ and  $n\in B\Rightarrow f(n)\in S_R$. Since $S$ is Boolean recursively isomorphic to $T$ via $g$, we have $f(n)\in S_P\Leftrightarrow g(f(n))\in T_P$ and $f(n)\in S_R \Leftrightarrow g(f(n))\in T_R$. By Theorem \ref{EI thm}, $T$ is {\sf EI}.

Suppose $S$ is {\sf TP}  and $S$ is Boolean recursively isomorphic to $T$ via a recursive bijection $g$. We show that $T$ is {\sf TP}. Suppose $U$ is a consistent RE extension of $T$. Then $g^{-1}[U]$ is a consistent RE extension of $g^{-1}[T]$. Since $g^{-1}[T]=S$ and $S$ is {\sf TP}, we have $S {\sf \leq_T} g^{-1}[U]$. Thus, $T {\sf \leq_T} U$.

(2): From Theorem \ref{EI BRI}, all {\sf EI} theories are Boolean recursively isomorphic. But ``$T$ is {\sf EI}" does not imply that $T$ interprets $\mathbf{R}$  from Theorem \ref{EI not R}. Thus, Boolean recursive isomorphisms do not preserve the property of interpreting $\mathbf{R}$.

From Theorem \ref{BRI HU}, Boolean recursive isomorphisms do not preserve {\sf EHU} theories.

Let $P$ be any one of the following properties: {\sf Rosser},
{\sf REW}, {\sf RFD}, {\sf RSS} and {\sf RSW}. From Theorem \ref{EI does not imply REW}, {\sf EI} does not imply the property $P$. Take an {\sf EI} theory $T$ such that $T$ does not have the property $P$. From Theorem \ref{property of R}, the theory $\mathbf{R}$ has the property $P$. Since both $\mathbf{R}$ and $T$ are {\sf EI} theories, by Theorem \ref{EI BRI}, $\mathbf{R}$ and $T$ are Boolean recursively isomorphic. Since $T$ does not have the property $P$, Boolean recursive isomorphisms do not preserve the property $P$.
\end{proof}

Now, we show that {\sf Creative} does not imply ${\sf EU}$.\smallskip

\begin{theorem}\label{creative not imply EU}
{\sf Creative} does not imply ${\sf EU}$.
\end{theorem}
\begin{proof}\label{}
Suppose $X$ is creative.  Define $T:={\sf J}+ \{A_n: n\in X\}$ (see Remark \ref{theory of J} for the definitions of ${\sf J}$ and $A_n$). We first show that $T$ is {\sf Creative}. By Fact  \ref{fact on creative}(1), it suffices to show that any RE set is reducible to $T_P$ via a recursive function. Let $B$ be an RE set. Since $X$ is creative, there exists a recursive function $f$ such that $n\in B \Leftrightarrow f(n)\in X$. Define the function $g: n\mapsto \ulcorner A_n\urcorner$. Let $h(n)=g(f(n))$. Note that $h$ is recursive. If $n\in B$, then $h(n)\in T_P$. Now suppose $h(n)\in T_P$. Then $T\vdash A_{f(n)}$.  We have $f(n)\in X$: by Theorem \ref{J thm}(2), $A_{f(n)}$ is equivalent with a Boolean combination of $A_i$'s over ${\sf J}$ with $i\in X$; if $f(n)\notin X$, this contradicts that $A_n$'s are mutually independent over ${\sf J}$.
Thus $n\in B$. We have $n\in B \Leftrightarrow h(n)\in T_P$. Hence, $T$ is {\sf Creative}.

Note that ${\sf J}+ \{A_n \mid n\in \omega\}$ is a consistent complete decidable RE extension of $T$ by Theorem \ref{J thm}.   Thus, $T$ is not {\sf EU}.
\end{proof}

Now we show that the property {\sf REW} does not imply {\sf EU}.\smallskip

\begin{theorem}[{\L}o\'{s}-Vaught test, \cite{Enderton}, p.157]~\label{Vaught test}
Let $T$ be a theory in a countable language.
Assume that $T$ has no finite models.
If $T$ is $\kappa$-categorical for some infinite cardinal $\kappa$, then $T$ is
complete.
\end{theorem}
\smallskip

\begin{fact}[\cite{Murawski}, Corollary 3.1.10]~\label{EU CT}
Let $T$ be a consistent RE theory.
The theory $T$ is ${\sf EU}$ iff any consistent RE extension of $T$ over the same language is incomplete.
\end{fact}

\begin{theorem}\label{REW not imply EU}
The property {\sf REW} does not imply {\sf EU}.
\end{theorem}
\begin{proof}\label{}
Define the theory $T$ as follows. The language of $T$ consists of the language of {\sf Succ} plus a new binary predicate $P(x,y)$. Let $T: ={\sf Succ} + \{P(\overline{i},\overline{n}): n\in W_i\}$.  We first show that all RE sets are weakly representable in $T$.

Let $A=W_i$ be an RE set. By a simple model theoretic argument, we can show that $n\in A \Leftrightarrow T\vdash P(\overline{i},\overline{n})$. It suffices to show that if $T\vdash P(\overline{i},\overline{n})$, then $n\in A$. Suppose $T\vdash P(\overline{i},\overline{n})$, but $n\notin A$. Suppose $\{n\}=W_j$. Then $T\vdash P(\overline{i},\overline{n})$ since $n\in W_j$. Let $\mathcal{N}=\langle\mathbb{N}, 0,S\rangle$ be the standard model of {\sf Succ} where $S$ is the successor function on $\mathbb{N}$.
Define $P^{\mathcal{N}}(i,n) \Leftrightarrow n\in W_i$. Note that $\mathcal{N}\models T$. Since $\mathcal{N}\models P(\overline{i},\overline{n})\wedge P(\overline{j},\overline{n})$, we have $n\in W_i \cap W_j$ which is a contradiction.

Define the theory $S:={\sf Succ} + \forall x \forall y P(x,y)$. Note that $S$ is a consistent RE extension of $T$. Note that $S$ has only infinite models and all models of $S$ with uncountable cardinality $\kappa$ are isomorphic by Fact \ref{fact on Succ}. By Theorem \ref{Vaught test}, $S$ is complete. By Fact \ref{EU CT},  $T$ is not {\sf EU}.
\end{proof}

The following theorem is a summary of the relationships between {\sf Rosser} theories and the other properties following `{\sf Rosser}' in Remark \ref{convention}.

\begin{theorem}\label{thm on Rosser}
\begin{enumerate}[(1)]
  \item {\sf Rosser} $\Rightarrow {\sf EI}\Rightarrow {\sf RI} \Rightarrow {\sf EU}$.
  \item The property {\sf EI} does not imply {\sf Rosser}.
  \item The property {\sf RI} does not imply {\sf Rosser}.
  \item The property {\sf EU} does not imply {\sf Rosser}.
  \item {\sf Rosser} implies {\sf TP}.
  \item The property {\sf TP}  does not imply {\sf Rosser}.
  \item {\sf Rosser} does not imply {\sf EHU}.
  \item The property {\sf EHU} does not imply {\sf Rosser}.
\item {\sf Rosser} implies {\sf Creative}.
  \item {\sf Creative} does not imply {\sf Rosser}.
  \item {\sf Rosser} implies $\mathbf{0}^{\prime}$.
  \item The property $\mathbf{0}^{\prime}$ does not imply {\sf Rosser}.
  \item The property {\sf REW} does not imply {\sf Rosser}.
  \item {\sf Rosser} does not imply  {\sf RFD}.
  \item {\sf Rosser} implies  {\sf RSS}.
  \item {\sf Rosser} implies  {\sf RSW}.
  \item Any one of {\sf RFD}, {\sf RSS} and {\sf RSW} does not imply {\sf Rosser}.
\end{enumerate}
\end{theorem}
\begin{proof}\label{}
\begin{enumerate}[(1)]
  \item Follows from Theorem \ref{relation about EI}.
  \item Follows from Theorem \ref{EI does not imply REW}.
  \item Follows from (2) since {\sf EI} implies {\sf RI}.
  \item Follows from (2) since {\sf EI} implies {\sf EU}.
  \item Any consistent RE extension of a {\sf Rosser} theory is {\sf Rosser} and hence has  Turing degree $\mathbf{0}^{\prime}$.
  \item From Theorem \ref{thm on TP original}, a {\sf TP} theory may have Turing degree less than $\mathbf{0}^{\prime}$. But {\sf Rosser} theories have Turing degree $\mathbf{0}^{\prime}$.
\item  The theory {\sf E} in Example \ref{Putnam E}  is {\sf Rosser}, but it has a decidable sub-theory and hence is not {\sf EHU}.
  \item From Theorem \ref{EHU degree}, {\sf EHU} theories can have any  non-zero RE Turing degree. From Theorem \ref{relation about EI}, {\sf Rosser} theories have  Turing degree $\mathbf{0}^{\prime}$.
      \item From Theorem \ref{relation about EI}, {\sf Rosser} $\Rightarrow {\sf EI} \Rightarrow$ {\sf Creative}.
  \item From Theorem \ref{creative not imply EU}, {\sf Creative} does not imply {\sf EU}.  On the other hand, {\sf Rosser} implies {\sf EU}.
      \item From Theorem \ref{relation about EI}, {\sf Rosser} implies {\sf Creative} and {\sf Creative} implies $\mathbf{0}^{\prime}$.
  \item From Theorem \ref{creative not imply EU}, {\sf Creative} does not imply {\sf EU}.  On the other hand, {\sf Creative} implies $\mathbf{0}^{\prime}$ and {\sf Rosser} implies  {\sf EU}.
  \item  From Theorem \ref{REW not imply EU},  {\sf REW} does not imply {\sf EU}. On the other hand, {\sf Rosser} implies  {\sf EU}.
  \item Putnam's theory {\sf E} in Example \ref{Putnam E} is {\sf Rosser} but it is not {\sf RFD}.
  \item Follows from Lemma \ref{Rosser implies RSS}.
  \item Note that {\sf Rosser} implies {\sf RSS} and {\sf RSS} implies {\sf RSW}.
  \item From Theorem \ref{Shoefield theory} and Lemma \ref{RFD imply RSS}, any one of {\sf RFD}, {\sf RSS} and {\sf RSW} does not imply {\sf Creative}. On the other hand, {\sf Rosser} implies {\sf Creative}.
\end{enumerate}
\end{proof}

\begin{remark}
In this paper, given two properties $A$ and $B$ in Remark \ref{convention}, if the arrow line from $A$ to $B$ is black, this means that $A$ implies $B$; if the arrow line from $A$ to $B$ is red, this means that $A$ does not imply $B$; if the arrow line from $A$ to $B$ is green, this means that whether $A$ implies $B$ is not answered in this paper.
\end{remark}

The following diagram  is a summary of the relationships between {\sf Rosser} theories and the other properties in Remark \ref{convention}.

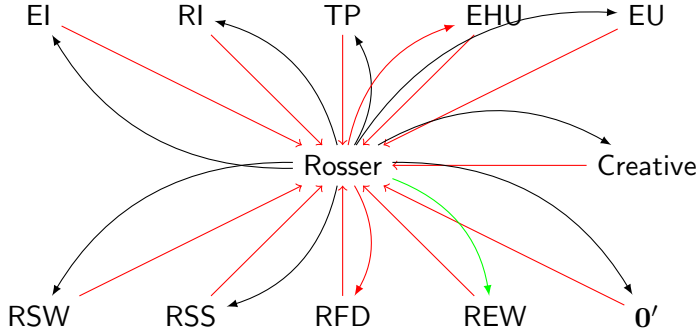
\begin{tikzpicture}

\begin{scope}[yshift=-1cm]

\node (1) at(4,2) {\sf Rosser};
\node (2) at(0,4) {\sf EI};
\node (3) at(2,4) {\sf RI};
\node (4) at(4,4) {\sf TP};

\node (5) at(6,4) {\sf EHU};
\node (6) at(8,4) {\sf EU};
\node (7) at(8,2) {\sf Creative};

\node (8) at(8,0) {$\mathbf{0}^{\prime}$};
\node (9) at(6,0) {\sf REW};
\node (10) at(4,0) {\sf RFD};
\node (11) at(2,0) {\sf RSS};

\node (12) at(0,0) {\sf RSW};

\draw[-latex,bend left]  (1) edge (2);
\draw[->, red] (2)-- (1);

\draw[-latex,bend right]  (1) edge (3);
\draw[->, red] (3)-- (1);

\draw[-latex,bend right]  (1) edge (4);
\draw[->, red] (4)--  (1);

\draw[-latex,bend left, red] (1) edge  (5);
\draw[->, red] (5)-- (1);

\draw[-latex,bend left]  (1) edge (6);
\draw[->, red] (6)--  (1);

\draw[-latex,bend left]  (1) edge (7);
\draw[->, red] (7)--  (1);

\draw[-latex,bend left]  (1) edge (8);
\draw[->, red] (8)--   (1);

\draw[-latex,bend left, green]  (1) edge (9);
\draw[->, red]  (9) --  (1);

\draw[-latex,bend left, red]  (1) edge (10);
\draw[->, red] (10)-- (1);

\draw[-latex,bend left]  (1) edge (11);
\draw[->, red] (11)-- (1);

\draw[-latex,bend right]  (1) edge (12);
\draw[->, red] (12)-- (1);
\end{scope}
\end{tikzpicture}

\section{Relationships with {\sf EI}  and {\sf RI} theories}

In this section, we discuss the relationships between {\sf EI} theories as well as {\sf RI} theories and the properties in Remark \ref{convention}.
\smallskip

\begin{definition}[\cite{Murawski}, Definition 2.4.18]~\label{}
We say $X\subseteq \mathbb{N}$  is \emph{universal} for recursive sets if for any recursive set $Y$, there is a recursive function $f$ such that $n\in Y\Leftrightarrow f(n)\in  X$.
\end{definition}
\smallskip

\begin{fact}[\cite{Murawski}, Lemma 2.4.19]~\label{universal recursive set}
There is no recursive set $X\subseteq \mathbb{N}$ that is universal for recursive sets.
\end{fact}

\begin{theorem}\label{RSS imply RI}
Let $T$ be a consistent RE theory. If $T$ is {\sf RSS}, then $T$ is {\sf RI}.
\end{theorem}
\begin{proof}\label{}
Suppose $T$ is {\sf RSS} but not {\sf RI}, i.e., there exists  a recursive set $X$ such that $T_P\subseteq X$ and $X\cap T_R=\emptyset$.
Now we show that $X$ is universal for the class of recursive sets. Let $Y$ be any recursive set.
Let $\phi(x)$ be a formula that strongly represents the recursive set $Y$ in $T$: if $n\in Y$, then $T\vdash \phi(\overline{n})$; and if $n\notin Y$, then $T\vdash \neg\phi(\overline{n})$. Define the function $f: n\mapsto \ulcorner \phi(\overline{n})\urcorner$. Note that $f$ is recursive. If $n\in Y$, then $f(n)\in T_P\subseteq X$. We show that if $f(n)\in X$, then $n\in Y$. Suppose $f(n)\in X$ but $n\notin Y$. Then $f(n)\in T_R\cap X$ which is a contradiction. Thus $X$ is universal for  recursive sets which contradicts Fact \ref{universal recursive set}.
\end{proof}

The following theorem is a summary of the relationships between {\sf EI} theories and the other properties following `{\sf EI}' in Remark \ref{convention}.

\begin{theorem}\label{thm on EI}
\begin{enumerate}[(1)]
  \item The property {\sf EI} implies  {\sf RI} and {\sf RI} implies  {\sf EU}.
  \item The property {\sf RI} does not imply {\sf EI}.
  \item The property {\sf EI} implies {\sf TP}.
  \item The property {\sf TP}  does not imply {\sf EI}.
  \item The property {\sf EI} does not imply {\sf EHU}.
  \item The property {\sf EHU} does not imply {\sf EI}.
  \item The property {\sf EU}  does not imply {\sf EI}.

  \item The property {\sf EI} implies {\sf Creative}.
  \item {\sf Creative} does not imply {\sf EI}.

  \item The property {\sf EI} implies $\mathbf{0}^{\prime}$.
  \item The property $\mathbf{0}^{\prime}$ does not imply {\sf EI}.

  \item  The property {\sf EI} does not imply {\sf REW}.
  \item The property {\sf REW} does not imply {\sf EI}.

  \item The property {\sf EI} does not imply  any one of {\sf RFD}, {\sf RSS} and {\sf RSW}.
  \item None of {\sf RFD}, {\sf RSS} and {\sf RSW} implies  {\sf EI}.
\end{enumerate}
\end{theorem}
\begin{proof}\label{}
\begin{enumerate}[(1)]
  \item Follows from Theorem \ref{relation about EI}.
  \item From Theorem \ref{Shoefield theory}, {\sf RFD} does not imply {\sf Creative}. From Lemma \ref{RFD imply RSS} and Theorem \ref{RSS imply RI},
      {\sf RFD} implies {\sf RI}. But {\sf EI} implies {\sf Creative}.
  \item Any consistent RE extension of an {\sf EI} theory has Turing degree $\mathbf{0}^{\prime}$.
  \item  A {\sf TP} theory may have Turing degree less than $\mathbf{0}^{\prime}$ from Theorem \ref{thm on TP original}. But {\sf EI} theories have  Turing degree $\mathbf{0}^{\prime}$.
      \item Note that {\sf EI} does not imply {\sf HU} since the theory $T$ in Theorem \ref{EI not R}(3) or the theory {\sf E} in Example \ref{Putnam E} is {\sf EI} but it has a decidable sub-theory and hence is not {\sf EHU}.
  \item From Theorem \ref{EHU degree}, {\sf EHU} theories can have any non-recursive RE  Turing degree. From Theorem \ref{relation about EI}, {\sf EI} theories have  Turing degree $\mathbf{0}^{\prime}$.
      \item From Theorem \ref{Shoefield theory} and ${\sf RFD} \Rightarrow {\sf RSS} \Rightarrow {\sf RI}\Rightarrow {\sf EU}$, we have {\sf EU} does not imply {\sf Creative}. On the other hand, {\sf EI} implies {\sf Creative}. Thus, {\sf EU} does not imply {\sf EI}.
  \item Follows from Theorem \ref{relation about EI}.
  \item From Theorem \ref{creative not imply EU}, {\sf Creative} does not imply {\sf EU}.  On the other hand,  {\sf EI} implies {\sf EU}.
  \item Follows from Theorem \ref{relation about EI}.
\item   From Theorem \ref{creative not imply EU}, {\sf Creative} does not imply {\sf EU}.  On the other hand, {\sf Creative} implies $\mathbf{0}^{\prime}$ and {\sf EI} implies {\sf EU}.
    \item Follows from Theorem \ref{EI does not imply REW}.
  \item From Theorem \ref{REW not imply EU},  {\sf REW} does not imply {\sf EU}. On the other hand,  {\sf EI} implies {\sf EU}.
  \item Follows from Theorem \ref{EI does not imply REW}.
  \item From Theorem \ref{Shoefield theory} and Lemma \ref{RFD imply RSS}, none of {\sf RFD}, {\sf RSS} and {\sf RSW} implies {\sf Creative}. On the other hand, {\sf EI} implies {\sf Creative}.
\end{enumerate}
\end{proof}

The following diagram  is a summary of the relationships between {\sf EI} theories and the other properties in Remark \ref{convention}.\smallskip

\begin{tikzpicture}

\begin{scope}[yshift=-1cm]

\node (1) at(4,2) {\sf EI};
\node (2) at(0,4) {\sf Rosser};
\node (3) at(2,4) {\sf RI};
\node (4) at(4,4) {\sf TP};

\node (5) at(6,4) {\sf EHU};
\node (6) at(8,4) {\sf EU};
\node (7) at(8,2) {\sf Creative};

\node (8) at(8,0) {$\mathbf{0}^{\prime}$};
\node (9) at(6,0) {\sf REW};
\node (10) at(4,0) {\sf RFD};
\node (11) at(2,0) {\sf RSS};

\node (12) at(0,0) {\sf RSW};

\draw[->, red] (1)-- (2);
\draw[-latex,bend right]  (2) edge (1);

\draw[-latex,bend left]  (1) edge (3);
\draw[->, red] (3)-- (1);

\draw[-latex,bend left]  (1) edge (4);
\draw[->, red] (4)--  (1);

\draw[-latex,bend left, red] (1) edge  (5);
\draw[->, red] (5)-- (1);

\draw[-latex,bend left]  (1) edge (6);
\draw[->, red] (6)--  (1);

\draw[-latex,bend left]  (1) edge (7);
\draw[->, red] (7)--  (1);

\draw[-latex,bend left]  (1) edge (8);
\draw[->, red] (8)--   (1);

\draw[-latex,bend left, red]  (1) edge (9);
\draw[->, red]  (9) --  (1);

\draw[-latex,bend left, red]  (1) edge (10);
\draw[->, red] (10)-- (1);

\draw[-latex,bend left, red]  (1) edge (11);
\draw[->, red] (11)-- (1);

\draw[-latex,bend right, red]  (1) edge (12);
\draw[->, red] (12)-- (1);
\end{scope}
\end{tikzpicture}
\smallskip

\begin{example}[\cite{PV}, Example 4.11]~\label{EU not imply TP}
There exists an {\sf EU} theory which is not {\sf TP}. To see this, suppose $X, Y$ are RE sets such that $X$ has Turing degree $\mathbf{d}$, $Y$ has Turing degree $\mathbf{e}$ and $\mathbf{0}<\mathbf{d}<\mathbf{e}$. Applying Theorem \ref{Shoenfield first} to $X$, let $(B,C)$ be the {\sf RI} pair as in Theorem \ref{Shoenfield first}. Define $T_0:= {\sf J} + \{A_{2n}: n\in B\}+\{\neg A_{2n}: n\in C\}$, $T_1:= T_0 + \{A_{2n+1}: n\in Y\}$. From Theorem \ref{thm on TP original}, $T_0$ is {\sf RI}, Turing persistent and has Turing degree $\mathbf{d}$. Note that $T_1$ is {\sf RI} and hence is {\sf EU}. We show that $T_1$ is not Turing persistent.

Note that $B,C,Y\leq_T T_1$ and $T_1\leq_T B,C,Y$ by an argument similar to the one in Theorem \ref{thm on TP original}. Since $\mathbf{d}<\mathbf{e}$, $T_1$ has  Turing degree $\mathbf{e}$. Let $T_2:= T_1 + \{A_{2n+1}: n\in \omega\}$. Note that $T_2:= {\sf J} + \{A_{2n}: n\in B\}+\{\neg A_{2n}: n\in C\}+ \{A_{2n+1}: n\in \omega\}$. By the similar argument as in Theorem \ref{thm on TP original}, $T_2$ has  Turing degree $\mathbf{d}$. Thus, $T_1$ is not Turing persistent since $T_2$ is a consistent RE extension of $T_1$ but $T_2 {\sf <_T} T_1$.
\end{example}

\begin{example}\label{Ehrenfeucht}
Ehrenfeucht constructs in \cite[pp.18-19]{Ehrenfeucht} an {\sf EU} theory which is not {\sf RI}.
\end{example}

\begin{theorem}\label{EHU not imply RI}
The property {\sf EHU} does not imply {\sf RI}.
\end{theorem}
\begin{proof}\label{}
From Example \ref{Ehrenfeucht}, take an {\sf EU} theory $T$ which is not {\sf RI}. Let $S:=pere(T)$ as in Theorem \ref{RBM key thm}. From Theorem \ref{RBM key thm} and Theorem \ref{Thm on EHU}, $S$ is {\sf EHU}.  Since $S$ is Boolean recursively isomorphic to $T$ and Boolean recursive isomorphisms preserve {\sf RI} theories, $S$  is not {\sf RI}.
\end{proof}

The following theorem is a summary of the relationships between {\sf RI} theories and the other properties following `{\sf RI}' in Remark \ref{convention}.

\begin{theorem}\label{thm on RI}
\begin{enumerate}[(1)]
\item  The property {\sf RI} does not imply {\sf TP}.

\item The property {\sf RI} does not imply {\sf EHU}.
\item The property {\sf EHU} does not imply {\sf RI}.

\item The property {\sf RI} implies {\sf EU}.
\item The property {\sf EU} does not imply ${\sf RI}$.

  \item The property {\sf RI} does not imply {\sf Creative}.
  \item  {\sf Creative} does not imply {\sf RI}.

  \item The property {\sf RI} does not imply $\mathbf{0}^{\prime}$.
  \item The property  $\mathbf{0}^{\prime}$ does not imply {\sf RI}.

\item The property {\sf RI} does not imply {\sf REW}.
\item The property {\sf REW} does not imply {\sf RI}.

\item The property {\sf RI} does not imply any one of {\sf RFD}, {\sf RSS} and {\sf RSW}.
\item The property {\sf RSS} implies {\sf RI}.
  \item The property {\sf RFD} implies {\sf RI}.
  \item The property {\sf RSW} does not imply {\sf RI}.
\end{enumerate}
\end{theorem}
\begin{proof}\label{}
\begin{enumerate}[(1)]
\item Note that the {\sf EU} theory which is not {\sf TP} in Example \ref{EU not imply TP} is also {\sf RI}.
\item Note that {\sf RI} does not imply {\sf HU} since the theory $T$ in Theorem \ref{EI not R}(2) or the theory {\sf E} in Example \ref{Putnam E} is {\sf RI} but it has a decidable sub-theory and hence is not {\sf EHU}.
\item  Follows from Theorem \ref{EHU not imply RI}.
\item Follows from Theorem \ref{relation about EI}.
\item Follows from Example \ref{Ehrenfeucht}.
\item From Theorem \ref{Shoefield theory}, {\sf RFD} does not imply {\sf Creative}. On the other hand, from Theorem \ref{RSS imply RI},  ${\sf RFD}\Rightarrow {\sf RSS} \Rightarrow {\sf RI}$.
\item From Theorem \ref{creative not imply EU}, {\sf Creative} does not imply {\sf EU}.  On the other hand, {\sf RI} implies {\sf EU}.
\item From Theorem \ref{thm on TP original}, {\sf RI} theories may have Turing degree less than $\mathbf{0}^{\prime}$.
\item From Theorem \ref{creative not imply EU}, {\sf Creative} does not imply {\sf EU}.  On the other hand,
 {\sf RI} implies {\sf EU} and {\sf Creative} implies $\mathbf{0}^{\prime}$.

\item From Theorem \ref{EI does not imply REW}, {\sf EI} does not imply {\sf REW}.  On the other hand, {\sf EI} implies {\sf RI}.
\item From Theorem \ref{REW not imply EU}, {\sf REW} does not imply {\sf EU}.  On the other hand, {\sf RI} implies {\sf EU}.
  \item From Theorem \ref{EI does not imply REW}, {\sf EI} does not imply any one of {\sf RFD}, {\sf RSS} and {\sf RSW}. On the other hand, {\sf EI} implies {\sf RI}.
  \item Follows from Theorem \ref{RSS imply RI}.
\item Note that {\sf RFD} implies {\sf RSS} and {\sf RSS} implies {\sf RI}.
\item  From Theorem \ref{REW not imply EU}, {\sf REW} does not imply {\sf EU}. On the other hand, {\sf REW} implies {\sf RSW} and {\sf RI} implies {\sf EU}.
\end{enumerate}
\end{proof}

The following diagram  is a summary of the relationships between {\sf RI} theories and the other properties in Remark \ref{convention}.
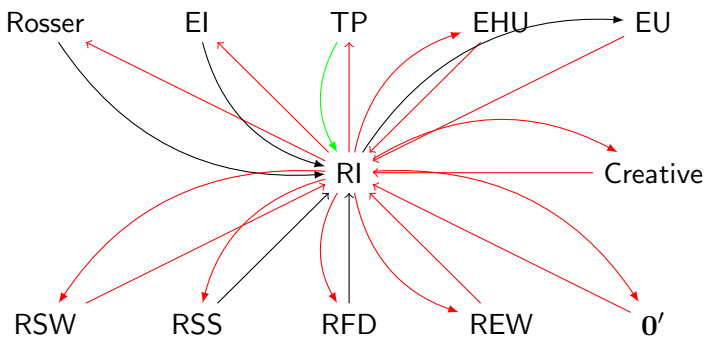
\smallskip

\begin{tikzpicture}

\begin{scope}[yshift=-1cm]

\node (1) at(4,2) {\sf RI};
\node (2) at(0,4) {\sf Rosser};
\node (3) at(2,4) {\sf EI};
\node (4) at(4,4) {\sf TP};

\node (5) at(6,4) {\sf EHU};
\node (6) at(8,4) {\sf EU};
\node (7) at(8,2) {\sf Creative};

\node (8) at(8,0) {$\mathbf{0}^{\prime}$};
\node (9) at(6,0) {\sf REW};
\node (10) at(4,0) {\sf RFD};
\node (11) at(2,0) {\sf RSS};

\node (12) at(0,0) {\sf RSW};

\draw[->, red] (1)-- (2);
\draw[-latex,bend right]  (2) edge (1);

\draw[->, red] (1)-- (3);
\draw[-latex,bend right]  (3) edge (1);

\draw[->, red] (1)--  (4);
\draw[-latex,bend right, green]  (4) edge (1);

\draw[-latex,bend left, red] (1) edge  (5);
\draw[->, red] (5)-- (1);

\draw[-latex,bend left]  (1) edge (6);
\draw[->, red] (6)--  (1);

\draw[-latex,bend left, red]  (1) edge (7);
\draw[->, red] (7)--  (1);

\draw[-latex,bend left, red]  (1) edge (8);
\draw[->, red] (8)--   (1);

\draw[-latex,bend right, red]  (1) edge (9);
\draw[->, red]  (9) --  (1);

\draw[-latex,bend right, red]  (1) edge (10);
\draw[->] (10)-- (1);

\draw[-latex,bend right, red]  (1) edge (11);
\draw[->] (11)-- (1);

\draw[-latex,bend right, red]  (1) edge (12);
\draw[->, red] (12)-- (1);
\end{scope}
\end{tikzpicture}

\section{Relationships with {\sf TP} and {\sf EHU} theories}

In this section, we discuss the relationships between {\sf TP} theories as well as  {\sf EHU} theories, and the other properties in Remark \ref{convention}.

\begin{lemma}\label{RSW imply undecidable}
Let $T$ be a consistent RE theory. If $T$  is  {\sf RSW}, then $T$ is undecidable.
\end{lemma}
\begin{proof}\label{}
Suppose that $T$ is a consistent RE
theory which is {\sf RSW} and decidable. Consider the following relation: $P(n,m)$ holds if and only if $n=\ulcorner\varphi(x)\urcorner$ and $T\vdash \varphi(\overline{m})$, where $\varphi(x)$ is a formula with exactly one free variable. Since $T$ is decidable, $P(n,m)$ is recursive. Define $D=\{n: \neg P(n,n)\}$. Then $D$ is recursive. Let $\psi(x)$ weakly represent  $D$ in $T$. Let $n=\ulcorner \psi(x)\urcorner$. Then $n \in D$ iff $\neg P(n,n)$ iff $T\nvdash\psi(\overline{n})$. On the other hand, $n \in D$ iff $T\vdash\psi(\overline{n})$, which is a contradiction.
\end{proof}

\begin{theorem}\label{REW not imply TP}
The property {\sf REW} does not imply {\sf TP}.
\end{theorem}
\begin{proof}\label{}
From Theorem \ref{REW not imply EU},  {\sf REW} does not imply {\sf EU}. Let $T$ be {\sf REW} but not {\sf EU}. We show that $T$ is not {\sf TP}. Suppose not, i.e., $T$ is  {\sf TP}. Since $T$ is not {\sf EU}, there exists a  consistent RE extension $S$ of $T$ such that $S$ is decidable. Since  $T$ is  {\sf TP}, we have $T {\sf \leq_T} S$. Thus, $T$ is decidable. By Lemma \ref{RSW imply undecidable}, since {\sf REW} implies {\sf RSW}, $T$ is undecidable which is a contradiction.
\end{proof}

The following theorem is a summary of the relationships between {\sf TP} theories and the other properties following `{\sf TP}' in Remark \ref{convention}.

\begin{theorem}\label{thm on TP}
\begin{enumerate}[(1)]
  \item The property {\sf TP} does not imply {\sf EHU}.
 \item The property {\sf EHU} does not imply {\sf TP}.

 \item The property {\sf TP} implies {\sf EU}.
  \item The property {\sf EU} does not imply {\sf TP}.

  \item The property {\sf TP}  does not imply {\sf Creative}.
  \item {\sf Creative} does not imply {\sf TP}.

  \item The property {\sf TP}  does not imply $\mathbf{0}^{\prime}$.
  \item The property $\mathbf{0}^{\prime}$ does not imply {\sf TP}.

  \item The property {\sf TP} does not imply {\sf REW}.
  \item The property {\sf REW} does not imply {\sf TP}.

  \item The property {\sf TP} does not imply any one of {\sf RFD}, {\sf RSS}, {\sf RSW}.
 \item The property {\sf RSW}  does not imply {\sf TP}.

\end{enumerate}
\end{theorem}
\begin{proof}\label{}
\begin{enumerate}[(1)]
\item The theory we construct in Theorem \ref{thm on TP original}  based on {\sf J} is {\sf TP} but not {\sf EHU} since it has a decidable sub-theory {\sf J}.
      \item From Example \ref{EU not imply TP}, take an {\sf EU} theory $T$ which is not {\sf TP}. Let $S:=pere(T)$  as in Theorem \ref{RBM key thm}. From Theorem \ref{RBM key thm} and Theorem \ref{Thm on EHU}, $S$ is {\sf EHU}.
          Since $S$ is  Boolean recursively isomorphic to $T$ and Boolean recursive isomorphisms preserve {\sf TP} theories,  $S$  is not {\sf TP}.
  \item Let $S$ be a consistent RE extension of $T$. Since $T$ has {\sf TP}, we have $S {\sf \geq_T} T$ and $T$ is undecidable. Thus, $S$ is undecidable.
      \item  Follows from Example \ref{EU not imply TP}.
  \item From Theorem \ref{thm on TP original}, a {\sf TP} theory may have the Turing degree less than $\mathbf{0}^{\prime}$.
  \item From Theorem \ref{creative not imply EU}, {\sf Creative} does not imply {\sf EU}. On the other hand, {\sf TP} implies {\sf EU}.
   \item From Theorem \ref{thm on TP original}, a {\sf TP} theory may have degree less than $\mathbf{0}^{\prime}$.
  \item Recall that in Example \ref{EU not imply TP} the theory $T_1:= T_0 + \{A_{2n+1}: n\in Y\}$ is {\sf EU} but not {\sf TP} where $Y$ is an RE set with Turing degree $\mathbf{e}>\mathbf{0}$. If we take $Y$ to be an RE set with Turing degree $\mathbf{0}^{\prime}$, then $T_1$ has Turing degree $\mathbf{0}^{\prime}$ but is not {\sf TP} as in Example \ref{EU not imply TP}.
  \item From Theorem \ref{EI does not imply REW}, {\sf EI} does not imply {\sf REW}. On the other hand, from Theorem \ref{thm on EI}(3), {\sf EI} implies {\sf TP}.
  \item Follows from Theorem \ref{REW not imply TP}.
  \item  From Theorem \ref{EI does not imply REW}, {\sf EI} does not imply any one of {\sf RFD}, {\sf RSS}, {\sf RSW}. On the other hand,  {\sf EI} implies {\sf TP}.
      \item From Theorem \ref{REW not imply TP}, {\sf REW} does not imply {\sf TP}. On the other hand,  {\sf REW} implies {\sf RSW}.
\end{enumerate}
\end{proof}

The following diagram  is a summary of the relationships between {\sf TP} theories and the other properties in Remark \ref{convention}.
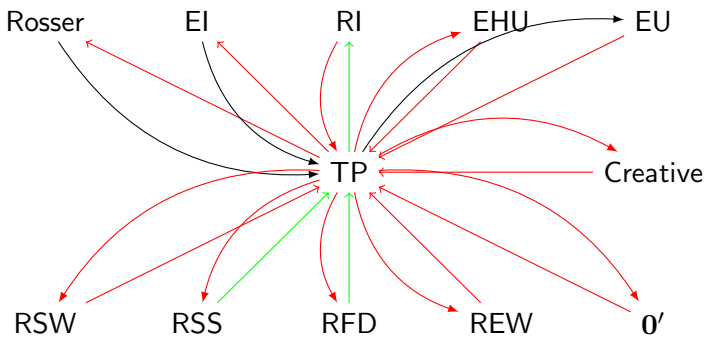
\smallskip

\begin{tikzpicture}

\begin{scope}[yshift=-1cm]

\node (1) at(4,2) {\sf TP};
\node (2) at(0,4) {\sf Rosser};
\node (3) at(2,4) {\sf EI};
\node (4) at(4,4) {\sf RI};

\node (5) at(6,4) {\sf EHU};
\node (6) at(8,4) {\sf EU};
\node (7) at(8,2) {\sf Creative};

\node (8) at(8,0) {$\mathbf{0}^{\prime}$};
\node (9) at(6,0) {\sf REW};
\node (10) at(4,0) {\sf RFD};
\node (11) at(2,0) {\sf RSS};

\node (12) at(0,0) {\sf RSW};

\draw[->, red] (1)-- (2);
\draw[-latex,bend right]  (2) edge (1);

\draw[->, red] (1)-- (3);
\draw[-latex,bend right]  (3) edge (1);

\draw[->, green] (1)--  (4);
\draw[-latex,bend right, red]  (4) edge (1);

\draw[-latex,bend left, red] (1) edge  (5);
\draw[->, red] (5)-- (1);

\draw[-latex,bend left]  (1) edge (6);
\draw[->, red] (6)--  (1);

\draw[-latex,bend left, red]  (1) edge (7);
\draw[->, red] (7)--  (1);

\draw[-latex,bend left, red]  (1) edge (8);
\draw[->, red] (8)--   (1);

\draw[-latex,bend right, red]  (1) edge (9);
\draw[->, red]  (9) --  (1);

\draw[-latex,bend right, red]  (1) edge (10);
\draw[->,green] (10)-- (1);

\draw[-latex,bend right, red]  (1) edge (11);
\draw[->,green] (11)-- (1);

\draw[-latex,bend right, red]  (1) edge (12);
\draw[->, red] (12)-- (1);
\end{scope}
\end{tikzpicture}

\begin{lemma}\label{REW imply creative}
The property {\sf REW} implies {\sf Creative}.
\end{lemma}
\begin{proof}\label{}
Suppose $T$ is {\sf REW}. Let $A$ be an RE set. Suppose $A$ is weakly representable in $T$ by $\phi(x)$. Let $f:n\mapsto \ulcorner \phi(\overline{n})\urcorner$. Note that $f$ is recursive and $n\in A \Leftrightarrow T\vdash \phi(\overline{n})\Leftrightarrow f(n)\in T_P$. From Fact \ref{fact on creative}(1), $T$ is {\sf Creative}.
\end{proof}

\begin{theorem}\label{RSS not imply EHU}
The property {\sf RSS} does not  imply {\sf EHU}.
\end{theorem}
\begin{proof}\label{}
Consider Putnam's theory {\sf E} in Example \ref{Putnam E}. We show that {\sf E} is {\sf RSS}. Let $A$ be a recursive set with $A=W_i$ and $\overline{A}=W_j$. From the axioms of {\sf E}, we have $n\in A\Rightarrow {\sf E}\vdash P_i(\overline{n})$. Suppose $n\notin A$. Then ${\sf E}\vdash P_j(\overline{n})$. Take the recursive function $e(i,j)$ as in Fact \ref{special rf}. Since $e(i,j)=i$ and $e(j,i)=j$ by Fact \ref{special rf}, we have  ${\sf E}\vdash P_i(\overline{n})\rightarrow \neg P_j(\overline{n})$. Thus, $n\notin A\Rightarrow {\sf E}\vdash \neg P_i(\overline{n})$. So $A$ is strongly representable in {\sf E}. Hence, {\sf E} is {\sf RSS}. Since {\sf E} has a decidable sub-theory, {\sf E} is not {\sf EHU}.
\end{proof}

The following theorem is a summary of the relationships between {\sf EHU} theories and the other properties following `{\sf EHU}' in Remark \ref{convention}.

\begin{theorem}\label{thm on EHU}
\begin{enumerate}[(1)]
  \item The property {\sf EHU} implies {\sf EU}.
  \item The property {\sf HU} does not imply {\sf EU}.
  \item The property {\sf EU} does not imply {\sf EHU}.

\item The property {\sf EHU}  does not imply {\sf Creative}.
\item  {\sf Creative} does not imply {\sf EHU}.

\item  The property {\sf EHU} does not imply $\mathbf{0}^{\prime}$.
  \item The property $\mathbf{0}^{\prime}$ does not imply {\sf EHU}.

\item The property {\sf EHU} does not imply {\sf REW}.
\item The property {\sf REW} does not imply {\sf EHU}.

  \item The property {\sf EHU} does not imply {\sf RSS}.
  \item  The property {\sf EHU} does not imply {\sf RFD}.

  \item The property {\sf RSS} does not  imply {\sf EHU}.
  \item The property {\sf RSW} does not  imply {\sf EHU}.
\end{enumerate}
\end{theorem}
\begin{proof}\label{}
\begin{enumerate}[(1)]
  \item     Follows from the definition of {\sf EHU}.
  \item Let $T$ be Robinson arithmetic $\mathbf{Q}$ without the first axiom. Since any theory consistent with $\mathbf{Q}$ over the same language  is undecidable, $T$ is {\sf HU} since any sub-theory of $T$ over the same language is consistent with $\mathbf{Q}$ and hence undecidable. Since $\mathbf{Q}$ is minimal essentially undecidable  by Fact \ref{fact on Q}, $T$ is not {\sf EU}.
 \item     The theory $T$ in Theorem \ref{EI not R}(3) or the theory {\sf E} in Example \ref{Putnam E} is {\sf EU} but it has a decidable sub-theory and hence is not {\sf EHU}.
  \item  From Theorem \ref{EHU degree}, {\sf EHU} theories can have any  non-recursive RE Turing degree. But {\sf Creative} theories have Turing degree $\mathbf{0}^{\prime}$.
  \item Let $T={\sf J}+ \{A_n: n\in X\}$ where $X$ is creative. It is easy to check that $T$ is {\sf Creative}, but $T$ is not {\sf EHU} since it has a decidable sub-theory. Another argument is: {\sf EI} does not imply {\sf EHU} and {\sf EI} implies {\sf Creative}.
\item From Theorem \ref{EHU degree}, {\sf EHU} theories can have any  non-recursive RE Turing degree.
  \item  Follows from (5) since {\sf Creative} implies $\mathbf{0}^{\prime}$.
\item Follows from (4) since {\sf REW} implies {\sf Creative} by Lemma \ref{REW imply creative}.
  \item From Theorem \ref{REW not imply R-like}, {\sf REW} does not imply {\sf HU}. On the other hand, {\sf EHU} implies {\sf HU}.
  \item  From Theorem \ref{EHU not imply RI},  {\sf EHU} does not imply {\sf RI}. From Theorem \ref{RSS imply RI},  {\sf RSS} implies {\sf RI}. Thus, {\sf EHU} does not imply {\sf RSS}.
      \item Follows from (10) since {\sf RFD} implies {\sf RSS}.
      \item Follows from Theorem \ref{RSS not imply EHU}.
      \item Follows from (12) since {\sf RSS} implies {\sf RSW}.
\end{enumerate}
\end{proof}

The following diagram  is a summary of the relationships between {\sf EHU} theories and the other properties in Remark \ref{convention}.
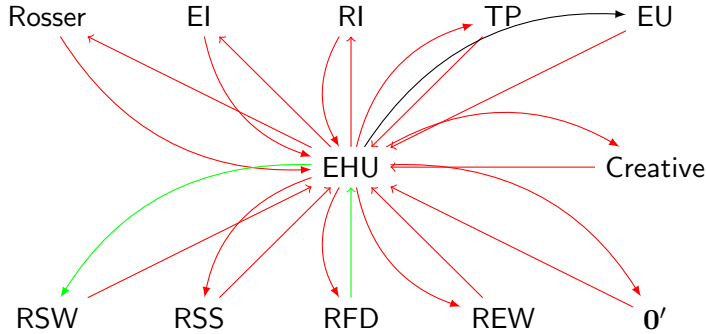
\smallskip

\begin{tikzpicture}

\begin{scope}[yshift=-1cm]

\node (1) at(4,2) {\sf EHU};
\node (2) at(0,4) {\sf Rosser};
\node (3) at(2,4) {\sf EI};
\node (4) at(4,4) {\sf RI};

\node (5) at(6,4) {\sf TP};
\node (6) at(8,4) {\sf EU};
\node (7) at(8,2) {\sf Creative};

\node (8) at(8,0) {$\mathbf{0}^{\prime}$};
\node (9) at(6,0) {\sf REW};
\node (10) at(4,0) {\sf RFD};
\node (11) at(2,0) {\sf RSS};

\node (12) at(0,0) {\sf RSW};

\draw[->, red] (1)-- (2);
\draw[-latex,bend right, red]  (2) edge (1);

\draw[->, red] (1)-- (3);
\draw[-latex,bend right, red]  (3) edge (1);

\draw[->, red] (1)--  (4);
\draw[-latex,bend right, red]  (4) edge (1);

\draw[-latex,bend left, red] (1) edge  (5);
\draw[->, red] (5)-- (1);

\draw[-latex,bend left]  (1) edge (6);
\draw[->, red] (6)--  (1);

\draw[-latex,bend left, red]  (1) edge (7);
\draw[->, red] (7)--  (1);

\draw[-latex,bend left, red]  (1) edge (8);
\draw[->, red] (8)--   (1);

\draw[-latex,bend right, red]  (1) edge (9);
\draw[->, red]  (9) --  (1);

\draw[-latex,bend right, red]  (1) edge (10);
\draw[->,green] (10)-- (1);

\draw[-latex,bend right, red]  (1) edge (11);
\draw[->, red] (11)-- (1);

\draw[-latex,bend right,green]  (1) edge (12);
\draw[->, red] (12)-- (1);
\end{scope}
\end{tikzpicture}

\section{Relationships with other properties}

In this section, we discuss the relationships between the properties in Remark \ref{convention} and the following properties in order: {\sf EU}, {\sf Creative}, $\mathbf{0}^{\prime}$, {\sf REW}, {\sf RFD}, {\sf RSS} and {\sf RSW}.\smallskip

\begin{fact}[\cite{Tarski}, Corollary 2, p.49; \cite{Murawski}, Theorem 2.4.20]~\label{fact on EU}
Either of {\sf RFD} and {\sf RSS} implies ${\sf EU}$.
\end{fact}
\begin{proof}\label{}
Follows from Theorem \ref{RSS imply RI} since ${\sf RFD}\Rightarrow {\sf RSS}$  and {\sf RI} implies {\sf EU}.
\end{proof}

The following theorem is a summary of the relationships between {\sf EU} theories and the other properties following `{\sf EU}' in Remark \ref{convention}.

\begin{theorem}\label{thm on EU}
\begin{enumerate}[(1)]
  \item The property ${\sf EU}$ does not imply {\sf Creative}.
  \item {\sf Creative} does not imply ${\sf EU}$.
      \item The property {\sf EU} does not imply  $\mathbf{0}^{\prime}$.
  \item The property $\mathbf{0}^{\prime}$ does not imply {\sf EU}.

  \item The property {\sf EU} does not imply {\sf REW}.
  \item The property {\sf REW} does not imply {\sf EU}.
      \item  The property ${\sf EU}$ does not imply any one of {\sf RFD}, {\sf RSS} and {\sf RSW}.
      \item Either of {\sf RFD} and {\sf RSS} implies ${\sf EU}$.
      \item The property {\sf RSW} does not imply ${\sf EU}$.
\end{enumerate}
\end{theorem}
\begin{proof}\label{}
\begin{enumerate}[(1)]
  \item From Theorem \ref{Shoefield theory}, {\sf RFD} does not imply {\sf Creative}. On the other hand, {\sf RFD} implies {\sf EU} by Fact \ref{fact on EU}.
  \item Follows from Theorem \ref{creative not imply EU}.
  \item From Theorem \ref{Shoenfield EU}, {\sf EU} theories may have Turing degree less than $\mathbf{0}^{\prime}$.
  \item From Theorem \ref{creative not imply EU}, {\sf Creative} does not imply {\sf EU}. On the other hand, {\sf Creative} implies $\mathbf{0}^{\prime}$.
  \item The theory in Theorem \ref{Shoefield theory} is {\sf EU}  but not {\sf REW} since no non-recursive set is weakly representable in it.
  \item Follows from Theorem \ref{REW not imply EU}.
  \item From Theorem \ref{EI does not imply REW}, ${\sf EI}$ does not imply any one of {\sf RFD}, {\sf RSS} and {\sf RSW}. On the other hand, ${\sf EI}$ implies ${\sf EU}$.
  \item Follows from Fact \ref{fact on EU}.
  \item From Theorem \ref{REW not imply EU}, {\sf REW} does not imply {\sf EU}. On the other hand, {\sf REW} implies {\sf RSW}.
\end{enumerate}
\end{proof}

The following diagram  is a summary of the relationships between {\sf EU} theories and the other properties in Remark \ref{convention}.
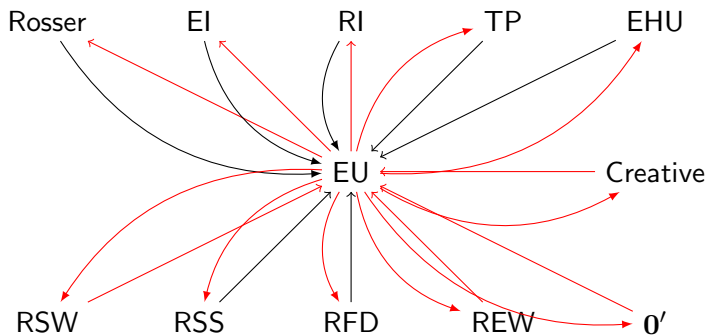
\smallskip

\begin{tikzpicture}

\begin{scope}[yshift=-1cm]

\node (1) at(4,2) {\sf EU};
\node (2) at(0,4) {\sf Rosser};
\node (3) at(2,4) {\sf EI};
\node (4) at(4,4) {\sf RI};

\node (5) at(6,4) {\sf TP};
\node (6) at(8,4) {\sf EHU};
\node (7) at(8,2) {\sf Creative};

\node (8) at(8,0) {$\mathbf{0}^{\prime}$};
\node (9) at(6,0) {\sf REW};
\node (10) at(4,0) {\sf RFD};
\node (11) at(2,0) {\sf RSS};

\node (12) at(0,0) {\sf RSW};

\draw[->, red] (1)-- (2);
\draw[-latex,bend right]  (2) edge (1);

\draw[->, red] (1)-- (3);
\draw[-latex,bend right]  (3) edge (1);

\draw[->, red] (1)--  (4);
\draw[-latex,bend right]  (4) edge (1);

\draw[-latex,bend left, red] (1) edge  (5);
\draw[->] (5)-- (1);

\draw[-latex,bend right, red]  (1) edge (6);
\draw[->] (6)--  (1);

\draw[-latex,bend right, red]  (1) edge (7);
\draw[->, red] (7)--  (1);

\draw[-latex,bend right, red]  (1) edge (8);
\draw[->, red] (8)--   (1);

\draw[-latex,bend right, red]  (1) edge (9);
\draw[->, red]  (9) --  (1);

\draw[-latex,bend right, red]  (1) edge (10);
\draw[->] (10)-- (1);

\draw[-latex,bend right, red]  (1) edge (11);
\draw[->] (11)-- (1);

\draw[-latex,bend right, red]  (1) edge (12);
\draw[->, red] (12)-- (1);
\end{scope}
\end{tikzpicture}
\smallskip

Now, we show that the property $\mathbf{0}^{\prime}$ does not imply {\sf Creative}.\smallskip

\begin{definition}[\cite{Rogers87}, pp.109-111]~\label{}
\begin{enumerate}[(1)]
  \item The ordered pair $\langle \langle x_1, \cdots, x_k\rangle, \alpha\rangle$ where $\langle x_1, \cdots, x_k\rangle$ is a $k$-tuple of integers and $\alpha$ is a $k$-ary Boolean function $(k>0)$ is called a \emph{truth-table condition}  (or \emph{{\sf tt}-condition}) of norm $k$.
  \item The {\sf tt}-condition $\langle \langle x_1, \cdots, x_k\rangle, \alpha\rangle$ is \emph{satisfied} by $A$ if $\alpha(c_A(x_1), \cdots, c_A(x_k))=1$ where $c_A$ is the characteristic function of $A$.
  \item Given $A, B\subseteq \mathbb{N}$, we say $A$ is \emph{truth-table reducible} to $B$ (denoted by $A\leq_{{\sf tt}} B$) if there exists a recursive function $f$ which always takes as image a {\sf tt}-condition such that for all $x, x\in A$ iff  the {\sf tt}-condition $f(x)$ is satisfied by $B$.
      \item We say $A\subseteq \mathbb{N}$ is \emph{truth-table complete} ({\sf tt}-complete) if $A$ is RE and for any RE set $B, B\leq_{{\sf tt}} A$.
          \item We say $A \equiv_{{\sf tt}} B$ if $A\leq_{{\sf tt}} B$ and $B\leq_{{\sf tt}} A$. The equivalence classes of $\equiv_{{\sf tt}}$ are called truth table degrees or ${\sf tt}$-degrees.
\end{enumerate}
\end{definition}
\smallskip

\begin{fact}[\cite{Rogers87}, p.112]~\label{tt-degree in zero sharp}
There are at least two truth table degrees in $\mathbf{0}^{\prime}$.
\end{fact}

\begin{theorem}\label{zero sharp does not imply creative}
There exists an RE theory which is $\mathbf{0}^{\prime}$ but not {\sf Creative}.
\end{theorem}
\begin{proof}\label{}
From Fact \ref{tt-degree in zero sharp}, suppose $X$ is in $\mathbf{0}^{\prime}$ but not in the unique {\sf tt}-degree of a creative set (i.e., not {\sf tt}-equivalent to any creative set). Let $T_X := {\sf J}+\{A_n \mid n\in X\}$. Note that $T_X$ is $\mathbf{0}^{\prime}$. We show that $T_X$ has the same {\sf tt}-degree as $X$. In the following, we identify $T_X$ with $(T_X)_P$.

Note that $X {\sf \leq_{tt}} T_X$ since $X$ is reducible to $T_X$ via the recursive function $g: n\mapsto \ulcorner A_n\urcorner$.
Now we show that $T_X {\sf \leq_{tt}} X$.

Given a sentence $\phi$, we want to know whether $T_X$ proves $\phi$. By Theorem \ref{J thm}, from $\phi$ we can effectively find $\phi^{\ast}$, a Boolean combination
of the $A_n$, such that ${\sf J}\vdash \phi\leftrightarrow \phi^{\ast}$. Let $\langle p_n: n\in\omega\rangle$ be the list of propositional variables, and $\psi$ be the  formula obtained by replacing $A_n$ with $p_n$ in $\phi^{\ast}$.

Note that the procedure to find $\psi$ starting from $\phi$ is effective: i.e., there exists a recursive function $f$ such that $f$ maps $\phi$ to $\psi$. We say that $\psi$ is satisfied by $X$ if $\psi$ has truth value 1 under the truth evaluation function $\sigma$ which satisfies $n\in X\Rightarrow  \sigma(p_n)=1$ and $n\notin X\Rightarrow  \sigma(p_n)=0$. Note that $T_X\vdash\phi$ iff $\psi$ is satisfied by $X$. This implies that $T_X {\sf \leq_{tt}} X$. Since $T_X$ has the same {\sf tt}-degree as $X$ and $X$ is not in the unique {\sf tt}-degree of a creative set, we have $T_X$ is not {\sf Creative}.
\end{proof}

The following theorem is a summary of the relationships between {\sf Creative} theories and the other properties following `{\sf Creative}' in Remark \ref{convention}.

\begin{theorem}\label{thm on creative}
\begin{enumerate}[(1)]
  \item {\sf Creative} implies $\mathbf{0}^{\prime}$.
  \item The property $\mathbf{0}^{\prime}$ does not imply {\sf Creative}.
  \item {\sf Creative} does not imply {\sf REW}.
  \item The property {\sf REW} implies {\sf Creative}.
  \item {\sf Creative} does not imply any one of {\sf RFD}, {\sf RSS} and {\sf RSW}.
  \item  None of {\sf RFD}, {\sf RSS} and {\sf RSW}
implies {\sf Creative}.
\end{enumerate}
\end{theorem}
\begin{proof}\label{}
\begin{enumerate}[(1)]
  \item Follows from Theorem  \ref{relation about EI}.
  \item Follows from Theorem \ref{zero sharp does not imply creative}.
  \item From Theorem  \ref{EI does not imply REW}, {\sf EI} does not imply {\sf REW}. On the other hand,  {\sf EI} implies {\sf Creative} by Theorem  \ref{relation about EI}.
  \item Follows from Lemma \ref{REW imply creative}.
   \item From Theorem  \ref{EI does not imply REW},  {\sf EI} does not imply any one of {\sf RFD}, {\sf RSS} and {\sf RSW}. On the other hand, {\sf EI} implies {\sf Creative} by Theorem  \ref{relation about EI}.
  \item Follows from Theorem \ref{Shoefield theory} and Lemma \ref{RFD imply RSS}.
\end{enumerate}
\end{proof}

The following diagram  is a summary of the relationships between {\sf Creative} theories and the other properties in Remark \ref{convention}.
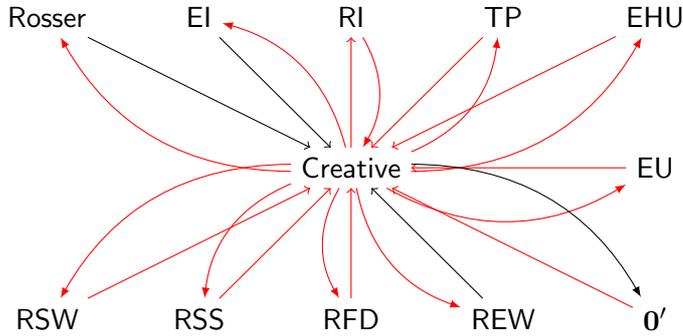
\smallskip

\begin{tikzpicture}

\begin{scope}[yshift=-1cm]

\node (1) at(4,2) {\sf Creative};
\node (2) at(0,4) {\sf Rosser};
\node (3) at(2,4) {\sf EI};
\node (4) at(4,4) {\sf RI};

\node (5) at(6,4) {\sf TP};
\node (6) at(8,4) {\sf EHU};
\node (7) at(8,2) {\sf EU};

\node (8) at(8,0) {$\mathbf{0}^{\prime}$};
\node (9) at(6,0) {\sf REW};
\node (10) at(4,0) {\sf RFD};
\node (11) at(2,0) {\sf RSS};

\node (12) at(0,0) {\sf RSW};

\draw[-latex,bend left, red] (1) edge (2);
\draw[->]  (2) -- (1);

\draw[-latex,bend right, red]  (1) edge (3);
\draw[->] (3)-- (1);

\draw[->, red] (1)--  (4);
\draw[-latex,bend left, red]  (4) edge (1);

\draw[-latex,bend right, red] (1) edge  (5);
\draw[->, red] (5)-- (1);

\draw[-latex,bend right, red]  (1) edge (6);
\draw[->, red] (6)--  (1);

\draw[-latex,bend right, red]  (1) edge (7);
\draw[->, red] (7)--  (1);

\draw[-latex,bend left]  (1) edge (8);
\draw[->, red] (8)--   (1);

\draw[-latex,bend right, red]  (1) edge (9);
\draw[->]  (9) --  (1);

\draw[-latex,bend right, red]  (1) edge (10);
\draw[->, red] (10)-- (1);

\draw[-latex,bend right, red]  (1) edge (11);
\draw[->, red] (11)-- (1);

\draw[-latex,bend right, red]  (1) edge (12);
\draw[->, red] (12)-- (1);
\end{scope}
\end{tikzpicture}
\smallskip

The following theorem is a summary of the relationships between $\mathbf{0}^{\prime}$ theories and the other properties following `$\mathbf{0}^{\prime}$' in Remark  \ref{convention}.

\begin{theorem}\label{thm on zero}
\begin{enumerate}[(1)]
  \item The property $\mathbf{0}^{\prime}$ does not imply {\sf REW}.
  \item The property {\sf REW} implies $\mathbf{0}^{\prime}$.
  \item The property $\mathbf{0}^{\prime}$ does not imply any one of {\sf RFD}, {\sf RSS}, {\sf RSW}.
  \item    None of {\sf RFD}, {\sf RSS}, {\sf RSW}
implies $\mathbf{0}^{\prime}$.
\end{enumerate}
\end{theorem}
\begin{proof}\label{}
\begin{enumerate}[(1)]
  \item From  Theorem  \ref{EI does not imply REW}, {\sf EI} does not imply {\sf REW}. On the other hand, {\sf EI} implies $\mathbf{0}^{\prime}$ by Theorem  \ref{relation about EI}.
  \item From Lemma \ref{REW imply creative}, {\sf REW} implies {\sf Creative}. On the other hand, {\sf Creative} implies $\mathbf{0}^{\prime}$.
  \item From Theorem  \ref{EI does not imply REW}, {\sf EI} does not imply any one of {\sf RFD}, {\sf RSS}, {\sf RSW}. On the other hand, {\sf EI} implies $\mathbf{0}^{\prime}$.
   \item   The theory in Theorem \ref{Shoefield theory} is {\sf RFD} (and hence {\sf RSS} and {\sf RSW}) but has Turing degree less than $\mathbf{0}^{\prime}$.
\end{enumerate}
\end{proof}

The following diagram  is a summary of the relationships between {\sf $\mathbf{0}^{\prime}$} theories and the other properties in Remark \ref{convention}.
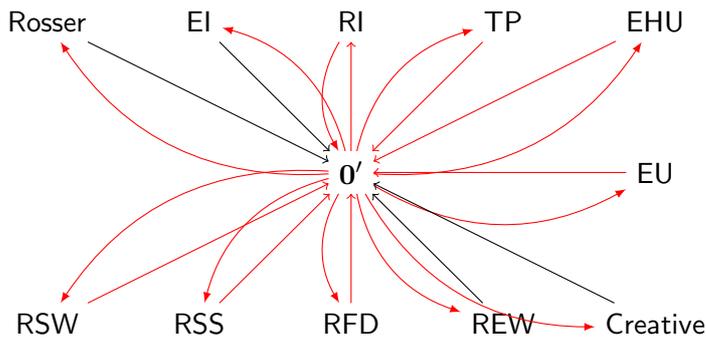
\smallskip

\begin{tikzpicture}

\begin{scope}[yshift=-1cm]

\node (1) at(4,2) {$\mathbf{0}^{\prime}$};
\node (2) at(0,4) {\sf Rosser};
\node (3) at(2,4) {\sf EI};
\node (4) at(4,4) {\sf RI};

\node (5) at(6,4) {\sf TP};
\node (6) at(8,4) {\sf EHU};
\node (7) at(8,2) {\sf EU};

\node (8) at(8,0) {\sf Creative};
\node (9) at(6,0) {\sf REW};
\node (10) at(4,0) {\sf RFD};
\node (11) at(2,0) {\sf RSS};

\node (12) at(0,0) {\sf RSW};

\draw[-latex,bend left, red] (1) edge (2);
\draw[->]  (2) -- (1);

\draw[-latex,bend right, red]  (1) edge (3);
\draw[->] (3)-- (1);

\draw[->, red] (1)--  (4);
\draw[-latex,bend right, red]  (4) edge (1);

\draw[-latex,bend left, red] (1) edge  (5);
\draw[->, red] (5)-- (1);

\draw[-latex,bend right, red]  (1) edge (6);
\draw[->, red] (6)--  (1);

\draw[-latex,bend right, red]  (1) edge (7);
\draw[->, red] (7)--  (1);

\draw[-latex,bend right, red]  (1) edge (8);
\draw[->] (8)--   (1);

\draw[-latex,bend right, red]  (1) edge (9);
\draw[->]  (9) --  (1);

\draw[-latex,bend right, red]  (1) edge (10);
\draw[->, red] (10)-- (1);

\draw[-latex,bend right, red]  (1) edge (11);
\draw[->, red] (11)-- (1);

\draw[-latex,bend right, red]  (1) edge (12);
\draw[->, red] (12)-- (1);
\end{scope}
\end{tikzpicture}
\smallskip

Finally,  we discuss the relationships between {\sf REW}, {\sf RFD}, {\sf RSS} as well as {\sf RSW} theories, and the properties following `{\sf REW}' in Remark \ref{convention}.

\begin{theorem}\label{thm on REW}
\begin{enumerate}[(1)]
  \item The property {\sf REW}  does not imply either of {\sf RFD} and {\sf RSS}.
  \item The property {\sf REW} implies {\sf RSW}.
  \item None of {\sf RFD}, {\sf RSS}, {\sf RSW} implies {\sf REW}.
\end{enumerate}
\end{theorem}
\begin{proof}\label{}
\begin{enumerate}[(1)]
\item From Theorem \ref{REW not imply EU}, {\sf REW} does not imply {\sf EU}. On the other hand, both {\sf RFD} and {\sf RSS} imply {\sf EU} by Fact \ref{fact on EU}.
\item Follows from Lemma \ref{Rosser implies RSS}(1).
\item The theory in Theorem \ref{Shoefield theory} is {\sf RFD} (and hence {\sf RSS} and {\sf RSW}) but it is not {\sf REW} since no non-recursive set is weakly representable in it.
\end{enumerate}
\end{proof}

The following diagram  is a summary of the relationships between {\sf REW} theories and the other properties in Remark \ref{convention}.
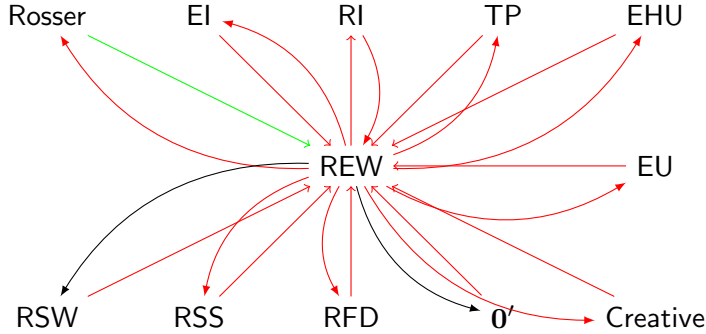
\smallskip

\begin{tikzpicture}

\begin{scope}[yshift=-1cm]

\node (1) at(4,2) {\sf REW};
\node (2) at(0,4) {\sf Rosser};
\node (3) at(2,4) {\sf EI};
\node (4) at(4,4) {\sf RI};

\node (5) at(6,4) {\sf TP};
\node (6) at(8,4) {\sf EHU};
\node (7) at(8,2) {\sf EU};

\node (8) at(8,0) {\sf Creative};
\node (9) at(6,0) {$\mathbf{0}^{\prime}$};
\node (10) at(4,0) {\sf RFD};
\node (11) at(2,0) {\sf RSS};

\node (12) at(0,0) {\sf RSW};

\draw[-latex,bend left, red] (1) edge (2);
\draw[->, green]  (2) -- (1);

\draw[-latex,bend right, red]  (1) edge (3);
\draw[->, red] (3)-- (1);

\draw[->, red] (1)--  (4);
\draw[-latex,bend left, red]  (4) edge (1);

\draw[-latex,bend right, red] (1) edge  (5);
\draw[->, red] (5)-- (1);

\draw[-latex,bend right, red]  (1) edge (6);
\draw[->, red] (6)--  (1);

\draw[-latex,bend right, red]  (1) edge (7);
\draw[->, red] (7)--  (1);

\draw[-latex,bend right, red]  (1) edge (8);
\draw[->, red] (8)--   (1);

\draw[-latex,bend right]  (1) edge (9);
\draw[->, red]  (9) --  (1);

\draw[-latex,bend right, red]  (1) edge (10);
\draw[->, red] (10)-- (1);

\draw[-latex,bend right, red]  (1) edge (11);
\draw[->, red] (11)-- (1);

\draw[-latex,bend right]  (1) edge (12);
\draw[->, red] (12)-- (1);
\end{scope}
\end{tikzpicture}

\begin{theorem}\label{}
\begin{enumerate}[(1)]
  \item The property {\sf RFD} implies {\sf RSS} and {\sf RSW}.
  \item The property {\sf RSS} does not  imply {\sf RFD}.
  \item The property {\sf RSS} implies {\sf RSW}.
  \item The property {\sf RSW} does not imply {\sf RSS}.
 \item  The property {\sf RSW} does not imply {\sf RFD}.
\end{enumerate}
\end{theorem}
\begin{proof}\label{}
\begin{enumerate}[(1)]
  \item Follows from Lemma \ref{RFD imply RSS}.
  \item Consider Putnam's theory {\sf E} in Example \ref{Putnam E}. From Theorem \ref{RSS not imply EHU}, {\sf E} is {\sf RSS}. But {\sf E} has only relation symbols without function symbols. Thus {\sf E} is not {\sf RFD}.
      \item Follows from the definition.
  \item From Theorem \ref{REW not imply EU}, {\sf REW} does not imply {\sf EU}. Let $T$ be a theory which is {\sf REW}
but not {\sf EU}. Since {\sf REW} implies {\sf RSW} and {\sf RSS} implies {\sf EU}, then  $T$ is {\sf RSW} but not {\sf RSS}.
      \item Follows from (4) since {\sf RFD} implies {\sf RSS}.
\end{enumerate}
\end{proof}

The following three diagrams  are a summary of the relationships among {\sf RFD},{\sf RSS} and {\sf RSW} theories and the other properties in Remark \ref{convention}.
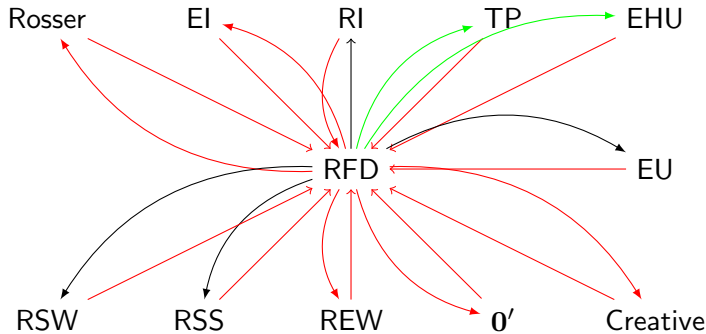
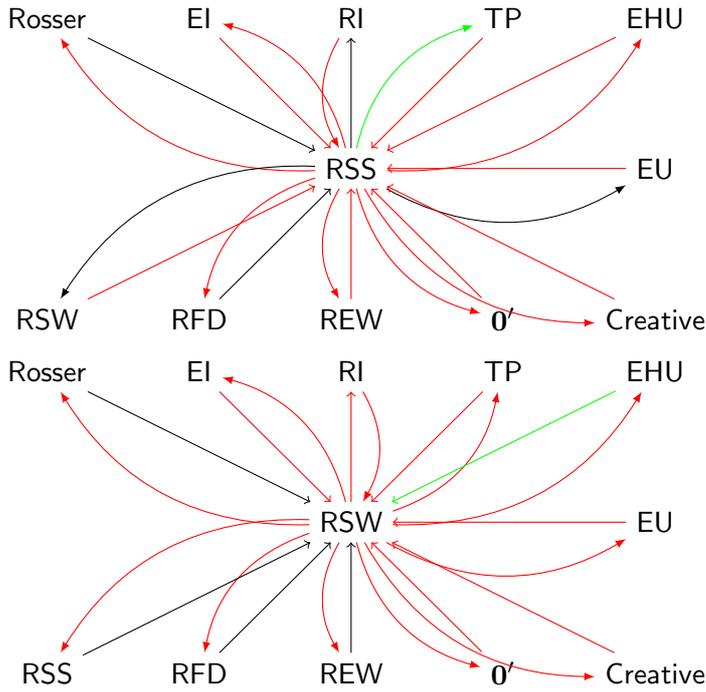
\smallskip

\begin{tikzpicture}
\begin{scope}[yshift=-1cm]

\node (1) at(4,2) {\sf RFD};
\node (2) at(0,4) {\sf Rosser};
\node (3) at(2,4) {\sf EI};
\node (4) at(4,4) {\sf RI};

\node (5) at(6,4) {\sf TP};
\node (6) at(8,4) {\sf EHU};
\node (7) at(8,2) {\sf EU};

\node (8) at(8,0) {\sf Creative};
\node (9) at(6,0) {$\mathbf{0}^{\prime}$};
\node (10) at(4,0) {\sf REW};
\node (11) at(2,0) {\sf RSS};

\node (12) at(0,0) {\sf RSW};

\draw[-latex,bend left, red] (1) edge (2);
\draw[->, red]  (2) -- (1);

\draw[-latex,bend right, red]  (1) edge (3);
\draw[->, red] (3)-- (1);

\draw[->] (1)--  (4);
\draw[-latex,bend right, red]  (4) edge (1);

\draw[-latex,bend left, green] (1) edge  (5);
\draw[->, red] (5)-- (1);

\draw[-latex,bend left, green]  (1) edge (6);
\draw[->, red] (6)--  (1);

\draw[-latex,bend left]  (1) edge (7);
\draw[->, red] (7)--  (1);

\draw[-latex,bend left, red]  (1) edge (8);
\draw[->, red] (8)--   (1);

\draw[-latex,bend right, red]  (1) edge (9);
\draw[->, red]  (9) --  (1);

\draw[-latex,bend right, red]  (1) edge (10);
\draw[->, red] (10)-- (1);

\draw[-latex,bend right]  (1) edge (11);
\draw[->, red] (11)-- (1);

\draw[-latex,bend right]  (1) edge (12);
\draw[->, red] (12)-- (1);
\end{scope}
\end{tikzpicture}
\smallskip

\begin{tikzpicture}
\begin{scope}[yshift=-1cm]

\node (1) at(4,2) {\sf RSS};
\node (2) at(0,4) {\sf Rosser};
\node (3) at(2,4) {\sf EI};
\node (4) at(4,4) {\sf RI};

\node (5) at(6,4) {\sf TP};
\node (6) at(8,4) {\sf EHU};
\node (7) at(8,2) {\sf EU};

\node (8) at(8,0) {\sf Creative};
\node (9) at(6,0) {$\mathbf{0}^{\prime}$};
\node (10) at(4,0) {\sf REW};
\node (11) at(2,0) {\sf RFD};
\node (12) at(0,0) {\sf RSW};

\draw[-latex,bend left, red] (1) edge (2);
\draw[->]  (2) -- (1);

\draw[-latex,bend right, red]  (1) edge (3);
\draw[->, red] (3)-- (1);

\draw[->] (1)--  (4);
\draw[-latex,bend right, red]  (4) edge (1);

\draw[-latex,bend left, green] (1) edge  (5);
\draw[->, red] (5)-- (1);

\draw[-latex,bend right, red]  (1) edge (6);
\draw[->, red] (6)--  (1);

\draw[-latex,bend right]  (1) edge (7);
\draw[->, red] (7)--  (1);

\draw[-latex,bend right, red]  (1) edge (8);
\draw[->, red] (8)--   (1);

\draw[-latex,bend right, red]  (1) edge (9);
\draw[->, red]  (9) --  (1);

\draw[-latex,bend right, red]  (1) edge (10);
\draw[->, red] (10)-- (1);

\draw[-latex,bend right, red]  (1) edge (11);
\draw[->] (11)-- (1);

\draw[-latex,bend right]  (1) edge (12);
\draw[->, red] (12)-- (1);
\end{scope}
\end{tikzpicture}
\smallskip

\begin{tikzpicture}
\begin{scope}[yshift=-1cm]

\node (1) at(4,2) {\sf RSW};
\node (2) at(0,4) {\sf Rosser};
\node (3) at(2,4) {\sf EI};
\node (4) at(4,4) {\sf RI};

\node (5) at(6,4) {\sf TP};
\node (6) at(8,4) {\sf EHU};
\node (7) at(8,2) {\sf EU};

\node (8) at(8,0) {\sf Creative};
\node (9) at(6,0) {$\mathbf{0}^{\prime}$};
\node (10) at(4,0) {\sf REW};
\node (11) at(2,0) {\sf RFD};
\node (12) at(0,0) {\sf RSS};

\draw[-latex,bend left, red] (1) edge (2);
\draw[->]  (2) -- (1);

\draw[-latex,bend right, red]  (1) edge (3);
\draw[->, red] (3)-- (1);

\draw[->, red] (1)--  (4);
\draw[-latex,bend left, red]  (4) edge (1);

\draw[-latex,bend right, red] (1) edge  (5);
\draw[->, red] (5)-- (1);

\draw[-latex,bend right, red]  (1) edge (6);
\draw[->, green] (6)--  (1);

\draw[-latex,bend right, red]  (1) edge (7);
\draw[->, red] (7)--  (1);

\draw[-latex,bend right, red]  (1) edge (8);
\draw[->, red] (8)--   (1);

\draw[-latex,bend right, red]  (1) edge (9);
\draw[->, red]  (9) --  (1);

\draw[-latex,bend right, red]  (1) edge (10);
\draw[->] (10)-- (1);

\draw[-latex,bend right, red]  (1) edge (11);
\draw[->] (11)-- (1);

\draw[-latex,bend right, red]  (1) edge (12);
\draw[->] (12)-- (1);
\end{scope}
\end{tikzpicture}

We conclude the paper with some questions that we did not answer in this paper. The theories in Questions (2)-(6) are understood
to be RE theories with numerals (the language of the theory has a constant $\mathbf{0}$ and a 1-ary function symbol $S$ such that we can define numerals in this language via $\mathbf{0}$ and $S$).\smallskip

\begin{question}
\begin{enumerate}[(1)]
  \item Does {\sf TP} imply {\sf RI}?
  \item Does Rosser imply {\sf REW}?
  \item Does {\sf RFD}  imply {\sf TP}?
  \item Does {\sf RSS} imply {\sf TP}?
  \item Does {\sf RFD} imply {\sf EHU}?
  \item Does {\sf EHU} imply {\sf RSW}?
\end{enumerate}
\end{question}

\end{document}